\documentclass[a4paper,12pt]{article}
\usepackage[text={6.5in,9in},centering]{geometry}
\usepackage{graphicx,url,subfig}
\usepackage{multirow}
\usepackage{booktabs}
\RequirePackage[OT1]{fontenc}
\RequirePackage{amsthm,amsmath,amssymb,amscd}
\RequirePackage[numbers]{natbib}
\RequirePackage[colorlinks,citecolor=blue,urlcolor
=blue]{hyperref}
\usepackage{enumerate}
\usepackage{datetime}
\usepackage{etex}
\makeatother
\numberwithin{equation}{section}
\allowdisplaybreaks
\usepackage{accents}
\newcommand{\ubar}[1]{\underaccent{\bar}{#1}}

\theoremstyle{plain}
\newtheorem{theorem}{Theorem}[section]

\newtheorem{lemma}[theorem]{Lemma}
\newtheorem{proposition}[theorem]{Proposition}

\theoremstyle{definition}
\newtheorem{definition}[theorem]{Definition}

\newtheorem{assumption}[theorem]{Assumption}

\newtheorem{remark}[theorem]{Remark}

\newtheorem{example}[theorem]{Example}

\newcommand{\E}{\mathbb{E}}
\newcommand{\W}{\dot{W}}
\newcommand{\ud}{\ensuremath{\mathrm{d}}}
\newcommand{\Ceil}[1]{\left\lceil #1 \right\rceil}
\newcommand{\Floor}[1]{\left\lfloor #1 \right\rfloor}

\newcommand{\Norm}[1]{\left|\left|  #1   \right|\right|}

\newcommand{\Id}{\text{Id}}

\newcommand{\Res}{\mathop{\text{Res}}}

\newcommand{\FoxH}[5]{H_{#2}^{#1}\left(#3\:\middle\vert\: \begin{subarray}{l}#4\\[0.4em] #5\end{subarray}\right)}

\newcommand{\calB}{\mathcal{B}}

\newcommand{\calF}{\mathcal{F}}
\newcommand{\calG}{\mathcal{G}}
\newcommand{\calK}{\mathcal{K}}
\newcommand{\calH}{\mathcal{H}}
\newcommand{\calL}{\mathcal{L}}
\newcommand{\calM}{\mathcal{M}}
\newcommand{\calN}{\mathcal{N}}

\newcommand{\calP}{\mathcal{P}}

\newcommand{\bbC}{\mathbb{C}}
\newcommand{\bbN}{\mathbb{N}}

\newcommand{\R}{\mathbb{R}}
\newcommand{\RR}{\mathbb{R}}

\newcommand{\Erfc}{\ensuremath{\mathrm{erfc}}}

\DeclareMathOperator{\Lip}{\mathit{L}}
\DeclareMathOperator{\LIP}{Lip}
\DeclareMathOperator{\lip}{\mathit{l}}

\DeclareMathOperator{\Vip}{\overline{\varsigma}}
\DeclareMathOperator{\vip}{\underline{\varsigma}}

\renewcommand{\Re}{\text{Re}}
\renewcommand{\Im}{\text{Im}}

\newcommand{\DCap}[1]{\partial^{#1}}
\newcommand{\DRL}[2]{{}_t D_{#1}^{#2}}
\newcommand{\lMr}[3]{\:{}_{#1} #2_{#3}}

\newcommand{\myRef}[2]{#1}

\title{Nonlinear stochastic time-fractional slow and fast diffusion equations on $\R^d$}

\author{\bf Le Chen
~~and~~Yaozhong Hu\footnote{Research
partially supported by a grant from the Simons Foundation
\#209206}~~and~~David Nualart\footnote{Research partially supported by the NSF grant  DMS1512891 and the ARO grant FED0070445.}\\[1em]
University of Kansas
\date{\small\today}
}

\begin{document}
\maketitle

\begin{center}
\begin{minipage}[rct]{5 in}
\footnotesize \textbf{Abstract:}
This paper studies the nonlinear stochastic partial differential equation of fractional orders
both in space and time variables:
\[
\left(\partial^\beta+\frac{\nu}{2}(-\Delta)^{\alpha/2}\right)u(t,x) = I_t^\gamma\left[\rho(u(t,x))\W(t,x)\right],\quad t>0,\: x\in\R^d,
\]
where $\W$ is the space-time white noise, $\alpha\in(0,2]$, $\beta\in(0,2)$, $\gamma\ge 0$ and $\nu>0$.
Fundamental solutions and their properties, in particular the nonnegativity, are derived.
The existence and uniqueness of solution together with the moment bounds of the solution are obtained under
Dalang's condition: $d<2\alpha+\frac{\alpha}{\beta}\min(2\gamma-1,0)$.
In some cases, the initial data can be measures. When $\beta\in (0,1]$, we prove the sample path regularity
of the solution. 

\vspace{2ex}
\textbf{MSC 2010 subject classifications:}
Primary 60H15. Secondary 60G60, 35R60.

\vspace{2ex}
\textbf{Keywords:}
nonlinear stochastic time-fractional diffusion equations,
measure-valued initial data,
H\"older continuity,
intermittency,
the Fox H-function.
\vspace{4ex}
\end{minipage}
\end{center}


\section{Introduction}
In this paper, we will study the following nonlinear stochastic time-fractional diffusion equations:
\begin{align}\label{E:SPDE}
  \begin{cases}
\left(\displaystyle \DCap{\beta} +\frac{\nu}{2}(-\Delta)^{\alpha/2} \right) u(t,x) =
I_t^\gamma\left[\rho\left(u(t,x)\right) \dot{W}(t,x) \right], & t>0,\:x\in\R^d.\\[0.2em]
u(0,\cdot)=\mu & \text{if $\beta\in \:(1/2,1]$,}\\[0.2em]
\displaystyle u(0,\cdot)=\mu_0,\quad\frac{\partial }{\partial t}u(0,\cdot)=\mu_1 & \text{if $\beta\in \:(1,2)\:$,}
\end{cases}
\end{align}
with $\alpha\in (0,2]$ and $\gamma>0$. In this equation,
$\Delta=\sum_{i=1}^d\partial^2/(\partial x_i^2)$ is the Laplacian with respect to the space variables and
$(-\Delta)^{\alpha/2}$ is the fractional Laplacian. $\W$ denotes the space-time white noise.
$\nu>0$ is the diffusion parameter.
The initial data $\mu$, $\mu_0$ and $\mu_1$
are assumed to be some measures. $\rho$ is a Lipschitz continuous function.
$\DCap{\beta}$ denotes
the {\it Caputo fractional differential operator}:
\[
\DCap{\beta} f(t) :=
\begin{cases}
\displaystyle
 \frac{1}{\Gamma(m-\beta)} \int_0^t\ud
\tau\: \frac{f^{(m)}(\tau)}{(t-\tau)^{\beta+1-m}}&
\text{if $m-1<\beta<m$\;,}\\[1em]
\displaystyle
\frac{\ud^m}{\ud t^m}f(t)& \text{if $\beta=m$}\;,
\end{cases}
\]
and $I_t^\gamma$ is the {\it Riemann-Liouville fractional integral} of order $\gamma> 0$:
\begin{align*}
I_t^\gamma f(t):= \frac{1}{\Gamma(\gamma)}\int_0^t (t-s)^{\gamma-1}f(s)\ud s,  \quad \text{for $t>0$},
\end{align*}
with the convention $I_t^0=\Id$ (the identity operator).
We refer to \cite{Die04,Podlubny99FDE,SamkoKilbasMarichev93} for more
details of these fractional differential operators.

This paper is an extension of a recent work by the first author \cite{Chen14Time},
where the case $\gamma=\Ceil{\beta}-\beta$, $\alpha=2$ and $d=1$ is studied.
Here, $\Ceil{\beta}$ is the smallest integer not less than $\beta$.
The interested reader can find motivations of the model in that reference.
The fractional integral operator $I_t^\gamma$ smooths the noise term.
Removing this integral operator by setting $\gamma=0$, one may expect that the solution becomes less regular.
As proved in \cite{Chen14Time},
when $\alpha=2$, $d=1$ and $\gamma=\Ceil{\beta}-\beta$, there is a mild solution for all $\beta\in(0,2)$.
This is no longer true if this fractional integral operator is not there.
In particular, we will show that, when $\alpha=2$, $d=1$ and $\gamma=0$, the mild solution exists only for $\beta\in(2/3,2)$ instead,
which is a direct consequence of condition \eqref{E:Dalang} below.

Motivations for stochastic partial differential equations (spde) with time-fractional derivative
can be found in many recent papers \cite{Chen14Time,ChenKimKim15,HuHu15,MijenaNane14ST}.
For convenience, we call equation \eqref{E:SPDE} with $\beta\in (0,1]$ the {\it slow diffusion equation}, and
equation \eqref{E:SPDE} with $\beta\in(1,2)$ the {\it fast diffusion equation}.

When $d=1$, $\beta=2$, $\alpha=2$ and $\gamma=0$,
the spde \eqref{E:SPDE} reduces to the
{\it stochastic wave equation} (SWE) on $\R$:
\begin{align}
 \label{E:Wave}
\left(\frac{\partial^2}{\partial t^2}- \frac{\nu}{2}
\frac{\partial^2}{\partial x^2}\right) u(t,x) = \rho(u(t,x)) \W(t,x)\:,
\end{align}
with the speed of wave propagation $(\nu/2)^{1/2}$.
When $d=1$, $\beta=1$, $\alpha=2$ and $\gamma=0$,
the spde \eqref{E:SPDE} reduces to the
{\it stochastic heat equation} (SHE) on $\R$:
\begin{align}
 \label{E:Heat}
\left(\frac{\partial}{\partial t}- \frac{\nu}{2}
\frac{\partial^2}{\partial x^2}\right) u(t,x) = \rho(u(t,x)) \W(t,x)\:.
\end{align}
These two special cases have been studied carefully; see \cite{LeChen13Thesis, ChenDalang13Holder,ChenDalang14Wave,ChenDalang13Heat,ConusEct12Initial}.
The spde \eqref{E:SPDE} for $\beta\in(0,1]$ and $\nu=1-\beta$ has been recently studied in
\cite{MijenaNane14Int,MijenaNane14ST}.
When the noise does not depend on time, a similar model with a general elliptic operator has
been studied in \cite{HuHu15}.
Another related equation is the {\it stochastic fractional heat equation} (SFHE) on $\R$:
\begin{align}
 \label{E:FracHeat}
\left(\frac{\partial}{\partial t}- \lMr{x}{D}{\delta}^\alpha \right) u(t,x) = \rho(u(t,x)) \W(t,x)\:,
\end{align}
which has been studied recently in \cite{ChenDalang14FracHeat,ChenKim14Comparison};
see also \cite{DebbiDozzi05On,FoondunKhoshnevisan08Intermittence}.

\bigskip

All investigations on spde's of the above kinds  require a good study of the corresponding Green functions.
As proved below,
there is a triplet
\[
\left\{\: Z(t,x),\:Z^*(t,x)\:, Y(t,x): \:\:  (t,x)\in\R_+\times\R^d \: \right\},
\]
depending on the parameters $(\alpha,\beta,\gamma,\nu)$,
such that the solution to \eqref{E:SPDE} with $\rho(u(t,x))\W(t,x)$ replaced by a nice function $f(t,x)$ is
represented by 
\begin{align}
 u(t,x) &=
 \begin{cases}
(Z(t,\cdot)*\mu)(x) + \left(Y\star f\right)(t,x), &\text{if $\beta\in(0,1]$,}\\[0.5em]
(Z^*(t,\cdot)*\mu_0)(x) +(Z(t,\cdot)*\mu_1)(x) + \left(Y\star f\right)(t,x), &\text{if $\beta\in(1,2)$,}
 \end{cases}
\end{align}
where ``$*$'' denotes the convolution in the space variable:
\begin{align}
(Z(t,\cdot)*\mu)(x):= \int_{\R^d} Z(t,x-y)\mu(\ud y),
\end{align}
and ``$\star$'' denotes the convolution in both space and time variables:
\[
(Y\star f)(t,x):= \int_0^t\int_{\R^d}  Y(t-s,x-y)f(s,y)\ud s \ud y.
\]
These fundamental solutions are expressed using the {\it Fox H-function} \cite{KilbasSaigo04H}.
%

\bigskip
If we denote the solution to the homogeneous equation of \eqref{E:SPDE} by $J_0(t,x)$, i.e.,
\begin{align}
J_0(t,x)=
\begin{cases}
 \left(Z(t,\cdot)*\mu\right)(x) & \text{if $\beta\in(0,1]$,}\\[0.2em]
 \left(Z^*(t,\cdot)*\mu_0\right)(x)+\left(Z(t,\cdot)*\mu_1\right)(x) & \text{if $\beta\in(1,2)$,}\\[0.2em]
\end{cases}
\end{align}
then the rigorous meaning of \eqref{E:SPDE} is the following stochastic integral equation:
\begin{equation}
\label{E:WalshSI}
 \begin{aligned}
  u(t,x) &= J_0(t,x)+I(t,x),\quad\text{where}\cr
I(t,x) &=\iint_{[0,t]\times\R}
Y\left(t-s,x-y\right)\rho\left(u(s,y)\right)
W(\ud s,\ud y).
 \end{aligned}
\end{equation}
The stochastic integral in the above equation is in the sense of Walsh \cite{Walsh86}.

To establish the the existence and uniqueness of random field solutions to
\eqref{E:SPDE}, the first step is to check {\it Dalang's condition} \cite{Dalang99Extending}:
\[
\int_0^t\ud s\int_{\R^d} \ud y\: |Y(s,y)|^2<\infty,\quad\text{for all $t>0$,}
\]
which is equivalent to the following condition (see Lemma \ref{L:Dalang}):
\begin{align}\label{E:Dalang}
 d<2\alpha + \frac{\alpha}{\beta}\min(2\gamma-1,0)=:\Theta,
\end{align}
which is equivalent to 
\begin{align}
\label{E:Dalang2}
\beta+\gamma>\frac{1}{2}\left(1 + \frac{d\beta}{\alpha}\right)
\quad\text{and}\quad d<2\alpha.
\end{align}
Note that \eqref{E:Dalang2} implies that the space dimension should be less than or equal to $3$.
Among all possible cases in \eqref{E:Dalang}, the following two special cases have better properties:
\begin{align}\label{E:CaseA}
\gamma = 0 \quad&\text{or}\quad \alpha>d=1,\\
 \alpha>d&=1,\label{E:CaseB}
\end{align}
As shown in Lemma \ref{L:HAt0} and Remark \ref{R:YZero} below, under both conditions \eqref{E:Dalang} and \eqref{E:CaseA},
the function $Y(1,x)$ is bounded at $x=0$.
Moreover, under \eqref{E:Dalang} and \eqref{E:CaseB},
all functions $Z(1,x)$, $Z^*(1,x)$ and  $Y(1,x)$ are bounded at $x=0$.

\bigskip
We prove the existence and uniqueness of random field solutions to \eqref{E:WalshSI} in the following three cases:
%

{\noindent\bf Case I:~} If we assume only Dalang's condition \eqref{E:Dalang},
we prove the existence and uniqueness when the initial data are such that
\begin{align}\label{E:InitCt}
\sup_{(s,x)\in[0,t]\times\R^d} |J_0(s,x)|<\infty,\qquad\text{for all $t>0$},
\end{align}
which is satisfied, for example, when initial data are bounded measurable functions.

{\noindent\bf Case II:~}
Under both \eqref{E:Dalang} and \eqref{E:CaseA},
we obtain moment formulas that are similar to those in \cite{ChenDalang13Heat,Chen14Time,ChenDalang14FracHeat}.
The initial data satisfy \eqref{E:InitCt}.

{\noindent\bf Case III:~}
Under both \eqref{E:Dalang} and \eqref{E:CaseB}, the initial data can be measures.
Let $\calM(\R)$ be the set of signed (regular) Borel measures on $\R$.
For $x\in\R$, define an auxiliary function
\begin{align}\label{E:f}
f_{\beta}(\eta,x) :=
\exp\left(-\eta\: |x|^{1+\Floor{\beta}}\right),
\end{align}
where $\Floor{\beta}$ is the largest integer not greater than $\beta$.
Note that the difference between $\Ceil{\beta}$ and $\Floor{\beta}+1$ for $\beta\in(0,2)$
is only at $\beta=1$.
The initial data are assumed to be Borel measures such that
\begin{align}
\begin{cases}
\displaystyle
\left(|\mu|*f_{\beta}(\eta,\cdot)\right)(x)<\infty,&\text{for all $\eta>0$ and $x\in\R$, if $\alpha=2$},\\[1em]
\displaystyle
\sup_{y\in\R}\int_{\R}|\mu|(\ud x)\frac{1}{1+|x-y|^{1+\alpha}}<+\infty, & \text{if $\alpha\in (1,2)$,}
\end{cases}
\label{E:InitData}
\end{align}
where for any Borel measure $\mu$, $\mu=\mu_+-\mu_-$ is the the Jordan decomposition
and $|\mu|=\mu_++\mu_-$.
We use $\calM_{\alpha,\beta}(\R)$ to denote these measures.
In this case, we prove the existence and uniqueness of a solution to \eqref{E:FracHeat}
for all initial data from $\calM_{\alpha,\beta}(\R)$.

\bigskip

Here are some special cases:
\begin{enumerate}[(1)]
\item For \eqref{E:Heat}, i.e., $\alpha=2$, $\beta=1$ and $\gamma=0$, the set of admissible initial data studied
in \cite{ChenDalang13Heat} is $\calM_H(\R)$, which corresponds to $\calM_{2,1}(\R)$ in this paper.
 \item Under the condition that $d=1$, $\alpha=2$, $\beta\in(0,2)$, $\gamma=\Ceil{\beta}-\beta$ (as in \cite{Chen14Time}),
 one can easily verify that condition \eqref{E:Dalang} is always true.
 The possible initial data is $\calM_T^\beta(\R)$, which corresponds to
$\calM_{2,\beta}(\R)$ in this paper.
 \item If $\gamma=1-\beta$ and $\beta\in (0,1)$, then it is ready to see that \eqref{E:Dalang}
 reduces to
 \[
 d< \alpha \min(2,\beta^{-1}),
 \]
 which recovers the condition by Mijena and Nane \cite{MijenaNane14ST}.
 \item If  $\gamma=0$, then \eqref{E:Dalang} becomes
 \[
 \frac{d}{\alpha}+\frac{1}{\beta}<2.
 \]
 Moreover, if $\alpha=2$ and $d=1$, then this condition becomes $\beta>2/3$, which coincides to the condition in \cite[Section 5.2]{ChenKimKim15}.
 \item If $\beta=1$ and $\gamma=0$, then Dalang's condition \eqref{E:Dalang} reduces to $\alpha>d$.
 Since $\alpha\in (0,2]$, we have that $\alpha\in (1,2]$ and $d=1$, which recovers the condition in \cite{ChenDalang14FracHeat}.
\end{enumerate}

As in \cite{ChenDalang14Wave,ChenDalang13Heat,ChenDalang14FracHeat},
we will obtain similar moment formulas expressed using a kernel
function $\calK(t,x)$ when \eqref{E:CaseA} is satisfied. For the SHE and the SWE,
this kernel function $\calK(t,x)$ has explicit forms.
But for the SFHE \cite{ChenDalang14FracHeat}, \eqref{E:SPDE} with $d=1$, $\gamma=\Ceil{\beta}-1$ and $\alpha=2$ in \cite{Chen14Time}, and
the current spde \eqref{E:SPDE},
we obtain some estimates on it.
In particular, we will obtain both upper and lower bounds on $\calK(t,x)$.

\bigskip
After establishing the existence and uniqueness of the solution,
we will study the sample-path regularity for the slow diffusion equations (i.e., the case when $\beta\in (0,1]$).
Given a subset $D\subseteq \R_+\times\R^d$ and positive constants
$\beta_1,\beta_2$, denote by $C_{\beta_1,\beta_2}(D)$ the set of functions
$v: \R_+ \times \R^d \mapsto \R$ with the following property:
for each compact set $K\subseteq D$,
there is a finite constant $C$ such that for all $(t,x)$ and $(s,y)\in K$,
\[
\vert v(t,x) - v(s,y) \vert \leq C \left(\vert
t-s\vert^{\beta_1} + \vert x-y \vert^{\beta_2}\right).
\]
Denote
\[
C_{\beta_1-,\beta_2-}(D) := \cap_{\alpha_1\in \;(0,\beta_1)}
\cap_{\alpha_2\in \;(0,\beta_2)} C_{\alpha_1,\alpha_2}(D)\;.
\]
We will show that for slow diffusion equations,
if the initial data has a bounded density, i.e., $\mu(\ud x)=f(x)\ud x$ with
$f\in L^\infty(\R^d)$, then
\begin{align}\label{E:Holder-u}
u(\cdot,\cdot) \in
C_{\frac{1}{2}\left(2(\beta+\gamma)-1-d\beta/\alpha\right)-,\:\frac{1}{2}\min(\Theta-d,2)-}\left(\R_+^*\times\R^d\right),\quad \text{a.s.},
\end{align}
where $\Theta$ is defined in \eqref{E:Dalang}, and $\R_+^*:=(0,\infty)$.
%

\bigskip
When the initial data are spatially homogeneous (i.e., the initial data are constants),
so is the solution $u(t,x)$,
and then the Lyapunov exponents
\begin{align}
 \overline{m}_p&:=\limsup_{t\rightarrow\infty}\frac{1}{t}\log \E\left[|u(t,x)|^p\right],\\
 \underline{m}_p&:=\liminf_{t\rightarrow\infty}\frac{1}{t}\log \E\left[|u(t,x)|^p\right],
\end{align}
do not depend on the spatial variable.
In this case, a solution is called {\it fully intermittent} if $m_1=0$ and
$\underline{m}_2>0$ (see \cite[Definition III.1.1, on p. 55]{CarmonaMolchanov94}).
As for the weak intermittency, there are various definitions.
For convenience of stating our results, we will call the solution {\it weakly intermittent of type I} if $\underline{m}_2>0$,
and {\it weakly intermittent of type II} if $\overline{m}_2>0$.
Clearly, the weak intermittency of type I is stronger than the
the weak intermittency of type II, but weaker than the full intermittency by missing $m_1=0$.
The weak intermittency of type II is used in \cite{FoondunKhoshnevisan08Intermittence}.

The full intermittency for the SHE and the SFHE are established in
\cite{BertiniCancrini94Intermittence} and \cite{ChenKim14Comparison}, respectively.
The weak intermittency of type I and II for SWE are
proved in \cite{ChenDalang14Wave} and \cite[Theorem 2.3]{ConusEtal13Wave}, respectively.
We will establish the weak intermittency of type II for both slow and fast diffusion equations.
For some slow diffusion equations, we will prove the weak intermittency of type I.
Moreover, we show that
\begin{align}\label{E:WI}
 \overline{m}_p\le C \: p^{1+\frac{1}{2(\beta+\gamma)-1-d\beta/\alpha}}\:.
\end{align}
It reduces to the following special cases:
\begin{enumerate}[(1)]
 \item The SHE case, i.e., $\beta =1$, $\alpha=2$, $\gamma=0$ and $d=1$: $\overline{m}_p \le C\: p^3$.
See \cite{BertiniCancrini94Intermittence,ChenDalang13Heat,FoondunKhoshnevisan08Intermittence};
\item The SWE case, i.e., $\beta=2$, $\alpha=2$, $\gamma=0$ and $d=1$:
$\overline{m}_p \le C\:p^{3/2}$. See \cite{ChenDalang14Wave};
\item The SFHE case, i.e., $\beta=1$, $\gamma=0$ and $d=1$:
$\overline{m}_p \le C\:p^{2+1/(\alpha-1)}$. See \cite{ChenDalang14FracHeat};
\item The time-fractional diffusion equation case as in \cite{Chen14Time} that $\alpha=2$, $\gamma=\Ceil{\beta}-\beta$ and $d=1$:
$\overline{m}_p \le C\:p^{\frac{4\Ceil{\beta}-\beta}{4\Ceil{\beta}-2-\beta}}$;
\item The time-fractional spde as in Theorems 3.11 and 3.12 of \cite{CHHH15Time} with $d=1$ and $\kappa=1$
\footnote{In \cite{CHHH15Time}, the constant $\kappa$ is the exponent for the Riesz kernel, and the case $\kappa=1$ corresponds to the space-time white noise.}:
\[
\begin{cases}
\overline{m}_p\le C \: p^{\frac{2\alpha-\beta}{2\alpha-\beta-\alpha}}& \text{when $\gamma=0$,}
\\[0.5em]
\overline{m}_p\le C \: p^{\frac{2\alpha\Ceil{\beta}-\beta}{2\alpha\Ceil{\beta}-\beta-\alpha}}& \text{when $\gamma=\Ceil{\beta}-\beta$.}
\end{cases}
\]
\end{enumerate}

\bigskip
In order to obtain some lower bounds for the moments, we prove that under the following cases
\begin{align}\label{E:4cases}
\begin{cases}
\text{Case I:~~} & \alpha\in (0,2],\: \beta\in(0,1),\:d\in\bbN,\: \gamma\ge 0, \cr
\text{Case II:~~} & \alpha\in (0,2],\: \beta=1,\: d\in\bbN, \gamma\in \{0\}\cup(1,\infty),\cr
\text{Case III:~~}& 1<\beta< \alpha \le 2,\: d\le 3,\: \gamma\ge 0,\cr
\text{Case IV:~~}& 1<\beta=\alpha < 2,\: d\le 3,\: \gamma\ge (d+3)/2-\beta,
\end{cases}
\end{align}
the fundamental solution $Y$ is nonnegative (see Theorem \ref{T:NonY} blow).
These results generalize those obtained by Mainardi {\it et al} \cite{MainardiEtc01Fundamental},
Pskhu \cite{Pskhu09}, and Chen {\it et al} \cite{CHHH15Time}.
In the end, we derive some lower bounds for the moments of the solution to \eqref{E:FracHeat} under \eqref{E:Dalang} and \eqref{E:4cases}.

\bigskip
This paper is structured as follows.
In Section \ref{S:Notation} we first give some notation and preliminaries.
The main results are stated in Section \ref{S:Main}.
The fundamental solutions are studied in Section \ref{S:Y}.
The proof of the two existence and uniqueness theorems are given in Section \ref{S:Proof}.
Finally, in Appendix, we prove some properties of the Fox H-functions.

\section{Some preliminaries and notation}\label{S:Notation}
Let $W=\left\{W_t(A):A\in\calB_b\left(\R^d\right),t\ge 0 \right\}$
be a space-time white noise
defined on a complete probability space $(\Omega,\calF,P)$, where
$\calB_b\left(\R^d\right)$ is the collection of Borel sets with finite Lebesgue measure.
Let
\[
\calF_t = \sigma\left(W_s(A):0\le s\le
t,A\in\calB_b\left(\R^d\right)\right)\vee
\calN,\quad t\ge 0,
\]
be the natural filtration augmented by the $\sigma$-field $\calN$ generated
by all $P$-null sets in $\calF$.
We use $\Norm{\cdot}_p$ to denote the $L^p(\Omega)$-norm ($p\ge 1$).
In this setup, $W$ becomes a worthy martingale measure in the sense of Walsh
\cite{Walsh86}, and $\iint_{[0,t]\times\R^d}X(s,y) W(\ud s,\ud y)$ is well-defined for a suitable
class of random fields $\left\{X(s,y),\; (s,y)\in\R_+\times\R^d\right\}$.
\bigskip

Recall that the rigorous meaning of the spde \eqref{E:SPDE} is in the integral form \eqref{E:WalshSI}.

\begin{definition}\label{DF:Solution}
A process $u=\left\{u(t,x),\:(t,x)\in\R_+^*\times\R^d\right\}$  is called a {\it random field solution} to
\eqref{E:SPDE} if
\begin{enumerate}[(1)]
 \item $u$ is adapted, i.e., for all $(t,x)\in\R_+^*\times\R^d$, $u(t,x)$ is $\calF_t$-measurable;
\item $u$ is jointly measurable with respect to
$\calB\left(\R_+^*\times\R^d\right)\times\calF$;
\item for all $(t,x)\in\R_+^*\times\R^d$, the following space-time convolution is finite:
\[
\left(Y^2 \star \Norm{\rho(u)}_2^2\right)(t,x):=
\int_0^t \ud s \int_{\R^d} \ud y
\: Y^2(t-s,x-y)\Norm{\rho(u(s,y))}_2^2
<+\infty;
\]
\item the function $(t,x)\mapsto I(t,x)$ mapping $\R_+^*\times\R^d$
into $L^2(\Omega)$ is continuous;
\item $u$ satisfies \eqref{E:WalshSI} a.s.,for all $(t,x)\in\R_+^*\times\R^d$.
\end{enumerate}
\end{definition}

Assume that the function $\rho:\R\mapsto \R$ is globally Lipschitz
continuous with Lipschitz constant $\LIP_\rho>0$.
We need some growth conditions on $\rho$ \footnote{This is a consequence of the Lipschitz continuity of $\rho$.}:
assume that for some constants $\Lip_\rho>0$ and
$\Vip \ge 0$,
\begin{align}\label{E:LinGrow}
|\rho(x)|^2 \le \Lip_\rho^2 \left(\Vip^2 +x^2\right),\qquad \text{for
all $x\in\R$}.
\end{align}
Sometimes we need a lower bound on $\rho(x)$:
assume that for
some constants $\lip_\rho>0$ and $\vip \ge 0$,
\begin{align}\label{E:lingrow}
|\rho(x)|^2\ge \lip_{\rho}^2\left(\vip^2+x^2\right),\qquad \text{for all $x\in\R$}\;.
\end{align}

For all $(t,x)\in\R_+^*\times \R^d$, $n\in\bbN$ and $\lambda\in\R$, define
\begin{align}
\notag
\calL_0\left(t,x\right) &:= Y^2(t,x)  \\
\label{E:Ln}
\calL_n\left(t,x\right)&:=
\left(\calL_0\star \dots\star\calL_0\right)(t,x),\quad\text{for $n\ge 1$, ($n$ convolutions),}\\
\calK\left(t,x;\lambda\right)&:= \sum_{n=0}^\infty
\lambda^{2(n+1)}  \calL_n\left(t,x;\lambda\right).
\label{E:K}
\end{align}
We will use the following conventions to the kernel functions $\calK(t,x;\lambda)$:
\begin{align}\label{E:Convention}
\begin{aligned}
 \calK(t,x) &:= \calK(t,x;\lambda),&
\overline{\calK}(t,x) &:= \calK\left(t,x;\Lip_\rho\right),\\
\underline{\calK}(t,x) &:= \calK\left(t,x;\lip_\rho\right),&
\widehat{\calK}_p(t,x) &:= \calK\left(t,x; 4\sqrt{p} \Lip_\rho\right),\quad\text{for $p\ge 2$}\:.
\end{aligned}
\end{align}
Throughout the paper, denote
\begin{align}\label{E:sigma}
\sigma:=2(1-\beta-\gamma) + \beta d/\alpha.
\end{align}
Note that
\begin{align}
\eqref{E:Dalang} \quad \Rightarrow \quad d<2\alpha +\frac{\alpha}{\beta}(2\nu-1) \quad \Leftrightarrow \quad
2(\beta+\gamma)-1-d\beta/\alpha>0
\quad
\Leftrightarrow
\quad \sigma <1.
\label{E:sigma<1}
\end{align}

Let $\DRL{+}{\alpha}$ denote the {\it Riemann-Liouville fractional derivative} of order $\alpha\in\R$
(see, e.g., \cite[(2.79) and (2.88)]{Podlubny99FDE}):
\begin{align}
\label{E:RLD}
\DRL{+}{\alpha} f(t):=
\begin{cases}
\displaystyle
 \frac{1}{\Gamma(m-\alpha)} \frac{\ud^m}{\ud t^m}\int_0^t\ud
\tau\: \frac{f(\tau)}{(t-\tau)^{\alpha+1-m}}&
\text{if $m-1<\alpha<m$ and $\alpha\ge 0$,}\\[1em]
\displaystyle
\frac{\ud^m}{\ud t^m}f(t)& \text{if $\alpha=m\ge 0$},\\[1em]
\displaystyle
\frac{1}{\Gamma(-\alpha)}\int_0^t s^{-\alpha-1}f(s)\ud s& \text{if $\alpha<0$}.
\end{cases}
\end{align}

We will need the {\it two-parameter Mittag-Leffler function}
\begin{align}\label{E:Mittag-Leffler}
E_{\alpha,\beta}(z) := \sum_{k=0}^{\infty}
\frac{z^k}{\Gamma(\alpha k+\beta)},
\qquad \alpha>0,\;\beta> 0\;,
\end{align}
which is a generalization of exponential function, $E_{1,1}(z)=e^z$; see, e.g.,
\cite[Section 1.2]{Podlubny99FDE}. A function is called {\it completely monotonic}
if  $(-1)^n f^{(n)}(x)\ge 0$ for $n =0,1,2,\dots$;
 see \cite[Definition 4]{Widder41LaplaceTr}.
An important fact \cite{Schneider96CM} concerning the Mittag-Leffler function is that
\begin{align}\label{E:E-CM}
 \text{$x\in\R_+\mapsto E_{\alpha,\beta}(-x)$ is
completely monotonic}
 \quad\Longleftrightarrow\quad
0<\alpha\le 1\wedge \beta.
\end{align}
By \cite[(2.9.27)]{KilbasSaigo04H}, the above Mittag-Leffler function is a special case of the Fox H-function:
\[
E_{\alpha,\beta}(z)=\FoxH{1,1}{1,2}{-z}{(0,1)}{(0,1),\:(1-\beta,\alpha)}.
\]

\section{Main results}\label{S:Main}
The first two theorems are about the existence, uniqueness and moment estimates of the solutions to \eqref{E:SPDE}.
The second one, in particular, possesses the same form as the one in \cite[Theorem 3.1]{Chen14Time}.
See also similar results for other equations, e.g., SHE \cite[Theorem 2.4]{ChenDalang13Heat},
SWE \cite[Theorem 2.3]{ChenDalang14Wave}, and
SFHE \cite[Theorem 3.1]{ChenDalang14FracHeat}.

\begin{theorem}[Existence, uniqueness and moments (I)]
\label{T2:ExUni}
Under \eqref{E:Dalang},
the spde \eqref{E:SPDE} has a unique (in the sense of versions) random field solution $\{u(t,x): (t,x)\in\R_+^* \times
\R^d \}$ if the initial data are such that
\begin{align}\label{E:BddInit}
\widehat{C}_t :=\sup_{(s,x)\in[0,t]\times\R^d} |J_0(s,x)|<+\infty.
\end{align}
Moreover, the following statements are true:
\begin{enumerate}[(1)]
 \item $(t,x)\mapsto u(t,x)$ is $L^p(\Omega)$-continuous for all $p\ge 2$;
\item For all even integers $p\ge 2$, all $t>0$ and
$x,y\in\R^d$,
\begin{align}\label{E:MomUpBdd}
 \Norm{u(t,x)}_p^2 \le
2J_0^2(t,x) + \left[\Vip^2+2\widehat{C}_t^2\right]\exp\left(C \lambda^{\frac{2}{1-\sigma}}t\right),
\end{align}
where $C$ is some universal constant not depending on $p$, and
$\sigma$ is defined in \eqref{E:sigma}.
\end{enumerate}
\end{theorem}
This theorem is proved in Section \ref{S2:ExUni}.
Note that if the initial data are bounded functions, then \eqref{E:BddInit} is satisfied.

\begin{theorem}[Existence, uniqueness and moments (II)]
\label{T:ExUni}
If Dalang's condition \eqref{E:Dalang} is satisfied, then the spde \eqref{E:SPDE} has a unique (in the sense of versions) random field solution $\{u(t,x): (t,x)\in\R_+^* \times
\R^d \}$ starting from
either initial data that satisfy \eqref{E:BddInit} under condition \eqref{E:CaseA}
or any Borel measures from $\calM_{\alpha,\beta}(\R)$ under condition \eqref{E:CaseB}.
Moreover, the following statements are true:
\begin{enumerate}[(1)]
 \item $(t,x)\mapsto u(t,x)$ is $L^p(\Omega)$-continuous for all $p\ge 2$;
\item For all even integers $p\ge 2$, all $t>0$ and
$x,y\in\R^d$,
\begin{align}\label{E:MomUp}
 \Norm{u(t,x)}_p^2 \le
\begin{cases}
 J_0^2(t,x) + \left(\left[\Vip^2+J_0^2\right] \star \overline{\calK} \right)
(t,x),& \text{if $p=2$}\;,\\[0.5em]
2J_0^2(t,x) + \left(\left[\Vip^2+2J_0^2\right] \star \widehat{\calK}_p \right)
(t,x),& \text{if $p>2$}\;;
\end{cases}
\end{align}
\item If $\rho$ satisfies \eqref{E:lingrow}, then under \eqref{E:Dalang}, \eqref{E:CaseA} and  the first two cases of \eqref{E:4cases},
for all $t>0$ and $x,y\in\R^d$, it holds that
\begin{align}
\label{E:SecMom-Lower}
\Norm{u(t,x)}_2^2 \ge J_0^2(t,x) + \left(\left(\vip^2+J_0^2\right) \star
\underline{\calK} \right)
(t,x)\;.
\end{align}
\end{enumerate}
\end{theorem}

The proof of this theorem is given in Section \ref{S:ExUni}.

\bigskip
The following theorem gives the H\"older continuity of the solution for slow diffusion equations.

\begin{theorem}
\label{T:Holder}
Recall that the constants $\sigma$ and $\Theta$ are defined in
\eqref{E:sigma} and \eqref{E:Dalang}, respectively.
If $\beta\in (0,1]$ and \eqref{E:BddInit} holds, then under \eqref{E:Dalang},
\begin{align}\label{E:LpBdd}
\sup_{(t,x)\in [0,T]\times\R^d}
\Norm{u(t,x)}_p^2 <+\infty, \quad\text{for all $T\ge 0$ and $p\ge 2$.}
\end{align}
Moreover, we have
\begin{align}\label{E:Holder-I}
I(\cdot,\cdot)\in
C_{\frac{1}{2}(1-\sigma)-,\: \frac{1}{2}\min(\Theta-d,2)-}\left(\R_+\times\R^d\right)\;,\quad\text{a.s.,}
\end{align}
and \eqref{E:Holder-u} holds.
\end{theorem}
\begin{proof}
The bound  \eqref{E:LpBdd} is due to \eqref{E:BddInit} and \eqref{E:MomUp}.
The proof of \eqref{E:Holder-I} is straightforward under \eqref{E:LpBdd} and Proposition \ref{P:G-SD} (see
\cite[Remark \myRef{4.6}{RH:BddHolder}]{ChenDalang13Holder}).
\end{proof}

In order to use the moment bounds in \eqref{E:MomUp} and \eqref{E:SecMom-Lower},
we need some good estimate on the kernel function $\calK(t,x)$.
Following \cite{Chen14Time}, define the following {\it reference kernel functions}:
\begin{align}
\calG_{\alpha,\beta}(t,x) :=
\begin{cases}\label{E:calG}
\displaystyle
c_\beta \left(4\nu\pi t^\beta\right)^{-d/2}
\exp\left(-\frac{1}{4\nu}\left(t^{-\beta/2}|x|\right)^{\Floor{\beta}+1}\right),
&\text{if $\alpha=2$,}\\[1em]
\displaystyle
\frac{c_d \:t^{\beta/\alpha}}{\left(t^{2\beta/\alpha}+|x|^2\right)^{(d+1)/2}},
&\text{if $\alpha\in (0,2)$,}
\end{cases}
\end{align}
for $\beta\in(0,2)$ where $|x|^2 = x_1^2+\dots +x_d^2$,
$c_\beta=1$ if $\beta\in [1,2)$ and $c_\beta=2^{-(1+d)}\nu^{-d/2}\Gamma(d/2)/\Gamma(d)$ if $\beta\in (0,1)$,
and $c_d=\pi^{-(d+1)/2}\Gamma((d+1)/2)$.
Define also
\begin{align}\label{E:calG-L}
\ubar{\calG}_{\alpha,\beta}(t,x):=
\begin{cases}
\displaystyle
\left(\nu\pi
t^\beta\right)^{-d/2} \exp\left(-\frac{|x|^2}{\nu t^\beta}\right)
& \text{if $\alpha=2$},\\[0.8em]
\displaystyle
\frac{c_d \:t^{\beta/\alpha}}{\left(t^{2\beta/\alpha}+|x|^2\right)^{(d+1)/2}},
&\text{if $\alpha\in (0,2)$.}
\end{cases}
\end{align}
These reference kernels are nonnegative and
the constants $c_\beta$ and $c_d$ are chosen such that the
integration of these kernels on $\R^d$ is equal to one.

\begin{theorem}
\label{T:K-bounds}
Fix $\lambda\in\R$.
\begin{enumerate}[(1)]
 \item Under \eqref{E:Dalang} and \eqref{E:CaseA}, there are two nonnegative constants $C$ and $\Upsilon$ depending on
 $\alpha$, $\beta$, $\gamma$, and $\nu$, such that, for all $(t,x)\in\R_+^*\times\R^d$,
\begin{align}\label{E:calK-U}
\calK(t,x;\lambda)\le \frac{C}{t^{\sigma}} \; \calG_{\alpha,\beta}(t,x) \left(1+ t^{\sigma}
\exp\left(\lambda^{\frac{2}{1-\sigma}} \: \Upsilon\: t\right)\right),
\end{align}
where $\sigma$ is defined in \eqref{E:sigma};
\item Under \eqref{E:Dalang}, \eqref{E:CaseA} and the first two cases in \eqref{E:4cases},
there are two nonnegative constants
$\ubar{C}$ and $\ubar{\Upsilon}$ depending on  $\alpha$, $\beta$, $\gamma$, and $\nu$,
such that, for all $(t,x)\in\R_+^*\times\R^d$,
\begin{align}
\label{E:calK-L}
\calK(t,x;\lambda)\ge \ubar{C} \; \ubar{\calG}_{\alpha,\beta}(t,x)
 \exp\left(\lambda^{\frac{2}{1-\sigma}} \:\ubar{\Upsilon}\: t\right).
\end{align}
\end{enumerate}
\end{theorem}
\begin{proof}
This theorem is due to Propositions \ref{P:ST-Con}, \ref{P:ST-Con2} and \cite[Proposition 5.2]{Chen14Time}.
\end{proof}

The last set of results are the weak intermittency.

\begin{theorem}[Weak intermittency]
\label{T:Weak-S}
Suppose that \eqref{E:Dalang} holds and the initial data satisfy \eqref{E:BddInit}.
\begin{enumerate}[(1)]
 \item If $\rho$ satisfies \eqref{E:LinGrow}, then
for some finite constant $C>0$,
\[
\overline{m}_p\le\:
C \Lip_\rho^{\frac{2}{2(\beta+\gamma)-1-d\beta/\alpha}}
p^{1+\frac{1}{2(\beta+\gamma)-1-d\beta/\alpha}},\quad\text{for all $p\ge 2$ even.}
\]
\item Suppose that the initial data are uniformly bounded from below,
i.e., $\mu(\ud x)=f(x)\ud x$ and $f(x)\ge c>0$ for all $x\in\R^d$.
If $\rho$ satisfies \eqref{E:lingrow} with $|c|+|\vip|\ne 0$,
then under \eqref{E:Dalang}, \eqref{E:CaseA} and
the first two cases in \eqref{E:4cases}, there is some finite constant $C'>0$ such that
\[
\underline{m}_p \ge C' \lip_\rho^{\frac{2}{2(\beta+\gamma)-1-d\beta/\alpha}} p
 ,\quad\text{for all $p\ge 2$.}
\]
\end{enumerate}
\end{theorem}
\begin{proof}
By \eqref{E:MomUp}, \eqref{E:calK-U} and \eqref{E:Csharp},
\[
\Norm{u(t,x)}_p^2 \le \widehat{C}_t^2 + C \: t^{-\sigma} \left(\Vip^2+2 \widehat{C}_t^2\right)
\Big(1+t^{\sigma} \exp\Big(\Upsilon \Lip_\rho^{\frac{2}{1-\sigma}} p^{\frac{1}{1-\sigma}}\:
t\Big)\Big).
\]
Then increase the power by a factor $p/2$. As for the lower bound, it holds that
\[
\Norm{u(t,x)}_p^2\ge
\Norm{u(t,x)}_2^2
\ge
c^2 + \ubar{C}  \left(\vip^2 + c^2\right)
 \exp\left( \lip_\rho^{\frac{2}{1-\sigma}} \ubar{\Upsilon} t
\right),
\]
thanks to \eqref{E:SecMom-Lower} and \eqref{E:calK-L}.
\end{proof}
\section{Fundamental solutions}\label{S:Y}

\begin{theorem}\label{T:PDE}
For $\alpha\in (0,2]$, $\beta\in (0,2)$ and $\gamma\ge 0$, the solution to
\begin{align} \label{E:PDE}
\begin{cases}
\displaystyle \left(\partial^\beta + \frac{\nu}{2} (-\Delta)^{\alpha/2} \right) u(t,x)= I_t^\gamma\left[f(t,x)\right],&\qquad t>0,\: x\in\RR^d, \\[0.5em]
\displaystyle \left.\frac{\partial^k}{\partial t^k} u(t,x)\right|_{t=0}=u_k(x), &\qquad 0\le k\le \Ceil{\beta}-1, \:\: x\in\RR^d,
\end{cases}
\end{align}
is
\begin{align}\label{E:Duhamel}
 u(t,x) = J_0(t,x) +
\int_0^t \ud s \int_{\RR^d} \ud y\: f(s,y) \: \DRL{+}{\Ceil{\beta}-\beta-\gamma} Z(t-s,x-y),
\end{align}
where 
\begin{align}\label{E:J0}
J_0(t,x):=
\sum_{k=0}^{\Ceil{\beta}-1}\int_{\RR^d} u_{\Ceil{\beta}-1-k}(y) \partial^k Z(t,x-y) \ud y
\end{align}
is the solution to the homogeneous equation and 
\begin{align}\label{E:Zab}
 Z_{\alpha,\beta,d}(t,x):= \pi^{-d/2} t^{\Ceil{\beta}-1} |x|^{-d}
 \FoxH{2,1}{2,3}{\frac{ |x|^\alpha}{2^{\alpha-1}\nu t^\beta}}{(1,1),\:(\Ceil{\beta},\beta)}
 {(d/2,\alpha/2),\:(1,1),\:(1,\alpha/2)},
\end{align}
\begin{align}
\label{E:Yab}
\hspace{-0.8 em}
Y_{\alpha,\beta,\gamma,d}(t,x) := \DRL{+}{\Ceil{\beta}-\beta-\gamma} Z_{\alpha,\beta,d}(t,x)
 = \pi^{-d/2} |x|^{-d}t^{\beta+\gamma-1}
 \FoxH{2,1}{2,3}{\frac{ |x|^\alpha}{2^{\alpha-1}\nu t^\beta}}
 {(1,1),\:(\beta+\gamma,\beta)}{(d/2,\alpha/2),\:(1,1),\:(1,\alpha/2)}
\end{align}
and, if $\beta\in(1,2)$,
\begin{align}\label{E:Z*ab}
 Z^*_{\alpha,\beta,d}(t,x):=\frac{\partial}{\partial t} Z_{\alpha,\beta,d}(t,x)=
 \pi^{-d/2} |x|^{-d}
 \FoxH{2,1}{2,3}{\frac{ |x|^\alpha}{2^{\alpha-1}\nu t^\beta}}{(1,1),\:(1,\beta)}
 {(d/2,\alpha/2),\:(1,1),\:(1,\alpha/2)}.
\end{align}
Moreover,
\begin{align}
 \label{E:FZ}
\calF Z_{\alpha,\beta,d}(t,\cdot)(\xi) &= t^{\Ceil{\beta}-1} E_{\beta,\Ceil{\beta}}(-2^{-1}\nu t^\beta |\xi|^\alpha),\\
 \calF Y_{\alpha,\beta,\gamma,d}(t,\cdot)(\xi)& = t^{\beta+\gamma-1} E_{\beta,\beta+\gamma}(-2^{-1}\nu t^\beta |\xi|^\alpha),
 \label{E:FY}\\
 \calF Z^*_{\alpha,\beta,d}(t,\cdot)(\xi) &= E_{\beta}(-2^{-1}\nu t^\beta |\xi|^\alpha), \quad \text{if $\beta\in (1,2)$}.
  \label{E:FZ*}
\end{align}
\end{theorem}
This theorem is proved in Section \ref{SS:PDE}.
For convenience, we will use the following notation
\begin{align}\label{E:YZ}
Y(t,x)&:=Y_{\alpha,\beta,\gamma,d}(t,x) = \DRL{+}{\Ceil{\beta}-\beta-\gamma} Z(t,x),\\
Z^*(t,x)&:=Z_{\alpha,\beta,d}^*(t,x) = \frac{\partial}{\partial t} Z(t,x),\quad \text{if $\beta\in (1,2)$}.
\label{E:Z*Z}
\end{align}
A direct consequence of expression \eqref{E:Yab} is the following scaling property
\begin{align}\label{E:YScale}
Y(t,x)=t^{\beta+\gamma-1-d\beta/\alpha} Y\left(1,t^{-\beta/\alpha} x\right).
\end{align}

\begin{remark}
By choosing $\alpha=2$, $d=1$ and $\beta$ arbitrarily close to $2$,
one can see that the first condition in \eqref{E:Dalang2} suggests the condition $\gamma>-1$.
However, when $\gamma\in (-1,0)$, one needs to specify another initial condition, namely, $\left.I_t^{1-\gamma} f(t,x)\right|_{t=0}$.
For example (Example 4.1 in \cite[p.138]{Podlubny99FDE}), the differential equation
\[
\DRL{+}{1/2}g(t)+g(t)=0, \qquad (t>0);  \qquad \left. I^{1/2}_t g(t)\right|_{t=0}=C,
\]
is solved by $g(t)=C\left(\frac{1}{\sqrt{\pi t}}-e^t\Erfc(\sqrt{t})\right)$.
This initial condition is obscure when the driving term $f$ becomes the multiplicative noise $\rho(u(t,x))\W(t,x)$.
Hence, throughout this paper, we assume $\gamma\ge 0$.
\end{remark}

\bigskip
This following lemma gives the asymptotics of fundamental solutions $Z_{\alpha,\beta,d}(1,x)$, $Y_{\alpha,\beta,\gamma,d}(1,x)$, and $Z_{\alpha,\beta,d}^*(1,x)$ at $x=0$
by choosing suitable values for $\eta$:
\[
\eta=
\begin{cases}
 \Ceil{\beta} &\text{in case of $Z$,}\cr
 \beta+\gamma &\text{in case of $Y$,} \cr
 1 &\text{in case of $Z^*$, when $\beta\in (1,2)$.}\cr
\end{cases}
\]
\begin{lemma}\label{L:HAt0}
Suppose $\alpha\in(0,2]$, $\beta\in(0,2)$, $\eta\in\R$, and $d\in\bbN$.
Let
\[
g(x)=x^{-d}\FoxH{2,1}{2,3}{x^\alpha}{(1,1),\;(\eta,\beta)}{(d/2,\alpha/2),\:(1,1),\:(1,\alpha/2)},\quad x>0.
\]
Then as $x\rightarrow 0_+$, the followings hold:
\[
g(x)
=
\begin{cases}
\frac{\Gamma(d-\alpha)/2}{\Gamma(\eta-\beta)\Gamma(\alpha/2)} \: x^{\alpha-d} + O(x^{\min(2\alpha-d,0)}) & \text{if $\eta\ne\beta$ and $d>\alpha$}: \hfill\text{Case 1,}\\[0.5em]
-\frac{\alpha}{\Gamma(\eta-\beta)\Gamma(1+d/2)}\: \log x +O(1) & \text{if $\eta\ne\beta$ and $d=\alpha$}: \hfill\text{Case 2,}\\[0.5em]
\frac{2}{\alpha} \frac{\Gamma(1-d/\alpha)\Gamma(d/\alpha)}{\Gamma(\eta-d\beta/\alpha)\Gamma(d/2)} +O(x^{\alpha-d}) & \text{if $\eta\ne\beta$ and $d<\alpha$}: \hfill\text{Case 3,}\\[0.5em]
\frac{2\Gamma(d/\alpha)}{\alpha\Gamma(d/2)} + O(x^2)& \text{if $\beta=\eta=1$}: \hfill\text{Case 4,}\\[0.5em]
 -\frac{\Gamma((d-2\alpha)/2)}{\Gamma(-\beta)\Gamma(\alpha)} \:x^{2\alpha-d} + O(x^{\min(3\alpha-d,0)})& \text{if $\beta=\eta\ne 1$ and $d/\alpha>2$}: \hfill\text{Case 5,}\\[0.5em]
 \frac{2\alpha}{\Gamma(-\beta)\Gamma(1+d/2)}\: \log x + O(x^{\alpha})& \text{if $\beta=\eta\ne 1$ and $d/\alpha=2$}: \hfill\text{Case 6,}\\[0.5em]
\frac{2}{\alpha} \frac{\Gamma(1-d/\alpha)\Gamma(d/\alpha)}{\Gamma(\beta(1-d/\alpha))\Gamma(d/2)}   + O(x^{2\alpha-d})& \text{if $\beta=\eta\ne 1$ and $d/\alpha\in (1,2)$}:\quad \hfill\text{Case 7,}\\[0.5em]
  -\frac{\beta}{\Gamma(1+d/2)} + O(x^{\alpha})& \text{if $\beta=\eta\ne 1$ and $d/\alpha=1$}: \hfill\text{Case 8,}\\[0.5em]
\frac{2}{\alpha} \frac{\Gamma(1-d/\alpha)\Gamma(d/\alpha)}{\Gamma(\beta(1-d/\alpha))\Gamma(d/2)}+ O(x^{\alpha})& \text{if $\beta=\eta\ne 1$ and $d/\alpha<1$}: \hfill\text{Case 9,}
\end{cases}
\]
where all the coefficients of the leading terms are finite and nonvanishing.
\end{lemma}
The calculations in the proof of this lemma is quite lengthy. We postpone it to Appendix \ref{SS:HAt0}.

\begin{remark}\label{R:YZero}
Since Dalang's condition \eqref{E:Dalang} implies $d<2\alpha$, the cases 5 and 6 are void under \eqref{E:Dalang}.
Combining the rest seven cases in Lemma \ref{L:HAt0}, we have that
\begin{align}\label{E:YZero}
\lim_{x\rightarrow 0} Y_{\alpha,\beta,\gamma,d}(1,x)
=
\begin{cases}
+\infty & \text{if $\gamma>0$ and $\alpha\le d<2\alpha$}:\hfill \text{Cases 1--2,}\\
C_1 &  \text{if $\gamma>0$ and $\alpha>d=1$}:\hfill \text{Case 3,}\\
C_2 &  \text{if $\gamma=0$, $\beta=1$ and $\alpha\ne d$}:\hfill \text{Case 4,}\\
C_3 &  \text{if $\gamma=0$, $\beta\ne 1$ and $d<2\alpha$}:\hfill \text{Cases 7--9,}
\end{cases}
\end{align}
and
\begin{align}\label{E:ZZero}
\lim_{x\rightarrow 0} Z_{\alpha,\beta,d}(1,x)
=
\begin{cases}
+\infty & \text{if $\beta\ne 1$ and $\alpha\le d<2\alpha$}:\hfill \text{Cases 1--2,}\\
C_4 &  \text{if $\beta\ne 1$ and $\alpha>d=1$}:\hfill \text{Case 3,}\\
C_2 &  \text{if $\beta=1$ and $\alpha\ne d$}:\hfill \text{Case 4,}\\
\end{cases}
\end{align}
and when $\beta\in (1,2)$,
\begin{align}\label{E:Z*Zero}
\lim_{x\rightarrow 0} Z_{\alpha,\beta,d}^*(1,x)
=
\begin{cases}
+\infty & \text{if $\alpha\le d<2\alpha$}:\hspace{3em} \text{Cases 1--2,}\\
C_5 &  \text{if $\alpha>d=1$}:\hfill \text{Case 3,}
\end{cases}
\end{align}
where the constants $C_i\in \R\setminus\{0\}$, $i=1,\dots,5$, only depend on $\alpha$, $\beta$, $\gamma$ and $d$.
Combining all these cases, we see that under \eqref{E:CaseA}, $Y(1,x)$ is bounded at $x=0$,
and under \eqref{E:CaseB}, all functions $Z(1,x)$, $Z^*(1,x)$ and $Y(1,x)$ are bounded at $x=0$.
\end{remark}

\begin{lemma}\label{L:YAsy}
$Y_{\alpha,\beta,\gamma,d}(1,x)$ has the following asymptotic property as $|x|\rightarrow \infty$:
\begin{align}\label{E:YAsy}
Y_{\alpha,\beta,\gamma,d}(1,x) \sim
\begin{cases}
 A_{\alpha} |x|^{-(d+\alpha)}& \text{if $\alpha\ne 2$},\\
 A_{2} |x|^{a}e^{-b|x|^c}
 &\text{if $\alpha=2$,}
\end{cases}
\end{align}
where the nonnegative constants are
\begin{align}\label{E:Aa}
A_{\alpha}=
\begin{cases}
 -\pi^{-d/2}\nu 2^{\alpha-1}\frac{\Gamma((d+\alpha)/2)}{\Gamma(2\beta+\gamma)\Gamma(-\alpha/2)} & \text{if $\alpha\ne 2$,}
 \\[0.5em]
\pi^{-d/2} (2-\beta)^{d/2-(\beta+\gamma)}\beta^{\frac{\beta(d+2-2(\beta+\gamma))}{2(2-\beta)}}(2\nu)^{\frac{2(\beta+\gamma)-(d+2)}{2(2-\beta)}}
&\text{if $\alpha=2$,}
\end{cases}
\end{align}
and
\begin{align}\label{E:abc}
a=\frac{d(\beta-1)-2(\beta+\gamma-1)}{2-\beta},\quad
b=(2-\beta)\beta^{\frac{\beta}{2-\beta}}(2\nu)^{\frac{1}{\beta-2}},
\quad\text{and}\quad
c=\frac{2}{2-\beta}.
\end{align}
Moreover, the asymptotic properties for $Z(1,x)$ and $Z^*(1,x)$ are the same as that for $Y(1,x)$ except that
the argument $\gamma$ in both \eqref{E:Aa} and \eqref{E:abc} should be replaced by $\Ceil{\beta}-\beta$ and $1$, respectively.
\end{lemma}
These asymptotics are obtained from \cite[Sections 1.5 and 1.7]{KilbasSaigo04H}.
We leave the details for interested readers.

\bigskip
\begin{theorem}\label{T:NonY}
Suppose that $\alpha\in(0,2]$, $\beta\in(0,2)$, and $\gamma\ge 0$.
The functions $Z(t,x):=Z_{\alpha,\beta,d}(t,x)$, $Y(t,x):=Y_{\alpha,\beta,\gamma,d}(t,x)$ and
$Z^*(t,x):=Z_{\alpha,\beta,d}^*(t,x)$, defined in Theorem \ref{T:PDE}, satisfy the following properties:
\begin{enumerate}[(1)]
\item For all $d\in\bbN$ and $\beta\in (0,1)$, both functions $Z$ and $Y$ are nonnegative.
When $\beta=1$, $Z$ is nonnegative, and $Y$ is nonnegative if either $\gamma=0$ or $\gamma>1$;
\item All functions $Z$, $Z^*$ and $Y$ are nonnegative if $d\le 3$ and $1<\beta<\alpha\le 2$.
When $1<\beta=\alpha<2$, $Y$ is nonnegative if $\gamma>(d+3)/2-\beta$;
\item When $d\ge 4$, $Y_{\alpha,\beta,0,d}(t,x)$ assumes both positive and negative values for all
$\alpha\in(0,2]$ and $\beta\in(1,2)$.
\end{enumerate}
\end{theorem}
This theorem is proved in Section \ref{SS:NonY}.
It generalizes the results by Mainardi {\it et al} \cite{MainardiEtc01Fundamental}
from one-space dimension to higher space-dimension.
Moreover, in \cite{MainardiEtc01Fundamental} only $Z$ when $\beta\in(0,1]$ and $Z^*$
when $\beta\in(1,2)$ are studied.
When $\beta\in (1,2)$, it also generalizes the results by Pskhu \cite{Pskhu09} from $\alpha=2$ and $\gamma=0$
to general $\alpha\in(0,2]$ and $\gamma>-1$.

\subsection{Some special cases}
In this part, we list some special cases.
\begin{example}
When $\gamma=0$ or $\gamma=\Ceil{\beta}-\beta$, the expressions for $Z$, $Y$ and $Z^*$ in Theorem \ref{T:PDE}
recover the results in \cite{CHHH15Time}.
\end{example}

\begin{example}
When $\alpha=2$, by \cite[Property 2.2]{KilbasSaigo04H}, we see that
\begin{align}\label{E:Z2b1}
 Z_{2,\beta,d}(t,x) &=
 \pi^{-d/2}t^{\Ceil{\beta}-1} |x|^{-d} \FoxH{2,0}{1,2}{\frac{|x|^2}{2\nu t^\beta}}{(\Ceil{\beta},\beta)}{(d/2,1),\:(1,1)},
\end{align}
and
\begin{align}\label{E:Y2b1}
 Y_{2,\beta,\gamma,d}(t,x) &=
 \pi^{-d/2} t^{\beta+\gamma-1}|x|^{-d}\FoxH{2,0}{1,2}{\frac{|x|^2}{2\nu t^\beta}}{(\beta+\gamma,\beta)}{(d/2,1),\:(1,1)},
\end{align}
and, when $\beta\in (1,2)$,
\begin{align}\label{E:Z*2b1}
 Z_{2,\beta,d}^*(t,x) &=
 \pi^{-d/2}|x|^{-d} \FoxH{2,0}{1,2}{\frac{x^2}{2\nu t^\beta}}{(1,\beta)}{(d/2,1),\:(1,1)}.
\end{align}
In particular, for $\beta\in(0,1)$ and $\gamma=0$, the expressions for $Z$ and $Y$ recover
those in \cite{Koc90,EK04}.
For $Z_{2,\beta,d}$, see also \cite[Chapter 6]{KilbasEtc06}.
When $\beta\in (1,2)$, $\gamma=0$ and $\nu=2$, the expression for $Y$ recovers the result in \cite{Pskhu09}.
\end{example}

\begin{example}
When $\alpha=2$ and  $d=1$, using Lemma \ref{L:H1/2} and \eqref{E:M-H2}, we see that
\begin{align}\label{E:Z2b1d1}
 Z_{2,\beta,1}(t,x) &=
 |x|^{-1} t^{\Ceil{\beta}-1} \FoxH{1,0}{1,1}{\frac{2x^2}{\nu t^\beta}}{(\Ceil{\beta},\beta)}{(1,2)}
 =\frac{t^{\Ceil{\beta}-1-\beta/2}}{\sqrt{2\nu}}
 M_{\beta/2,\Ceil{\beta}}\left(\frac{|x|}{\sqrt{\nu/2} \: t^{\beta/2}}\right),
\end{align}
and
\begin{align}\label{E:Y2b1d1}
 Y_{2,\beta,\gamma,1}(t,x) &=
 |x|^{-1} t^{\beta+\gamma-1}\FoxH{1,0}{1,1}{\frac{2x^2}{\nu t^\beta}}{(\beta+\gamma,\beta)}{(1,2)}=
 \frac{t^{\beta/2+\gamma-1}}{\sqrt{2\nu}}
 M_{\beta/2,\beta+\gamma}\left(\frac{|x|}{\sqrt{\nu/2} \: t^{\beta/2}}\right),
\end{align}
and, when $\beta\in (1,2)$,
\begin{align}\label{E:Z*2b1d1}
 Z_{2,\beta,1}^*(t,x) &=
 |x|^{-1} \FoxH{1,0}{1,1}{\frac{2x^2}{\nu t^\beta}}{(1,\beta)}{(1,2)}=
  \frac{t^{-\beta/2}}{\sqrt{2\nu}}
 M_{\beta/2,1}\left(\frac{|x|}{\sqrt{\nu/2} \: t^{\beta/2}}\right),
\end{align}
where $M_{\lambda,\mu} (z)$ is the {\it two-parameter Mainardi functions} (see \cite{Chen14Time}) of order
$\lambda\in [0,1)$,
\begin{align}\label{E:2p-Mainardi}
M_{\lambda,\mu} (z) :=\sum_{n=0}^\infty \frac{(-1)^n \: z^n}{n! \;\Gamma\left(
\mu-(n+1)\lambda\right)}
\;,\quad\text{for $\mu$ and $z\in\bbC$\:.}
\end{align}
For example, $M_{1/2,1}(z)=\frac{1}{\sqrt{\pi}}\exp\left(-z^2/4\right)$.
The {\em one-parameter Mainardi functions} $M_\lambda(z)$ are used by
Mainardi, {\it et al} in \cite{MainardiEtc01Fundamental,Mainardi10Book}.
\end{example}

\begin{example}
In \cite{MainardiEtc01Fundamental}, the fundamental solutions $Z_{\alpha,\beta,d}(t,x)$ for $\beta\in (0,1]$
and $Z^*_{\alpha,\beta,d}(t,x)$ for $\beta\in (1,2]$ have been studied for all $\alpha\in(0,2)$ and $d=1$.
From the Mellin-Barnes integral representation (6.6) of \cite{MainardiEtc01Fundamental}, we can see that the reduced Green function of \cite{MainardiEtc01Fundamental}
can be expressed using the Fox H-function:
\begin{align}\label{E:Kabt}
K_{\alpha,\beta}^\theta(x)=\frac{1}{|x|}
\FoxH{2,1}{3,3}{|x|^\alpha}
{(1,1),\:(1,\beta),\:(1,\frac{\alpha-\theta}{2})}
{(1,1),\:(1,\alpha),\:(1,\frac{\alpha-\theta}{2})}, \quad x\in\RR,
\end{align}
where $\alpha$ and $\beta$ have the same meaning as this paper and $\theta$ is the skewness: $|\theta|\le \min(\alpha,2-\alpha)$.
For the symmetric $\alpha$-stable case, i.e., $\theta=0$, this expression can be simplified using
Lemma \ref{L:H1/2}.
Hence,
\begin{align}
K_{\alpha,\beta}^0(x)=\frac{1}{\sqrt{\pi}|x|}
\FoxH{2,1}{2,3}{(|x|/2)^\alpha}
{(1,1),\:(1,\beta)}
{(1/2,\alpha/2),\:(1,1),\:(1,\alpha/2)},\quad x\in\RR.
\end{align}
Therefore, their fundamental solution \cite[(1.3)]{MainardiEtc01Fundamental}
\[
G^0_{\alpha,\beta}(x,t)=t^{-\beta/\alpha}K^0_{\alpha,\beta}(t^{-\beta/\alpha} x) =
\frac{1}{\sqrt{\pi}|x|}
\FoxH{2,1}{2,3}{\frac{|x|^\alpha}{2^\alpha t^\beta}}
{(1,1),\:(1,\beta)}
{(1/2,\alpha/2),\:(1,1),\:(1,\alpha/2)}
\]
corresponds, in the case when $\nu=2$, to our $Z_{\alpha,\beta,1}(t,x)$ when $\beta\in (0,1]$
and $Z^*_{\alpha,\beta,1}(t,x)$ when $\beta\in (1,2)$.
\end{example}

Here we draw some graphs\footnote{
The graphs are produced by truncating the infinite sum in \eqref{E:2p-Mainardi} by the first $24$ terms.
In Figure \ref{F:Ytx}, due to the bad approximations for small $t$ when truncating the infinite sum, the graphs are produced for $t$ staying away from $0$.
}
of these Green functions $Y(t,x)$: see Figures \ref{F:Yx} and \ref{F:Ytx}.
As $\beta$ approaches $2$, the graphs of $Y(t,x)$ become closer to the wave kernel $\frac{1}{2}1_{|x|\le \nu t/2}$.
\begin{figure}[h!tbp]
 \center
 \includegraphics[scale=0.6]{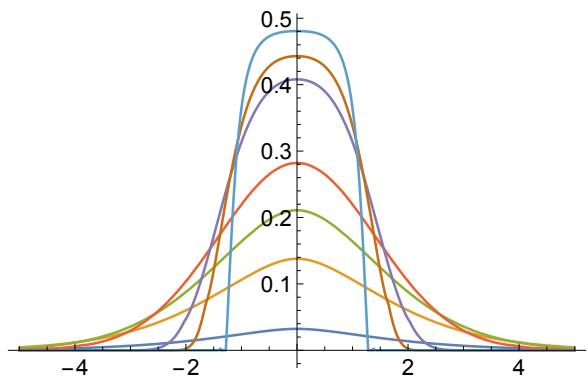}
 \caption{Some graphs of the function $Y_{2,\beta,0,1}(1,x)$ with $\nu=2$,
 and $\beta=15/8$, $5/3$, $3/2$, $1$, $3/4$, $1/2$, and $1/8$ from top to bottom.}
 \label{F:Yx}
\end{figure}

\begin{figure}[h!tbp]
\centering
\subfloat[$\beta=6/5$.]{
\includegraphics[scale=0.4]{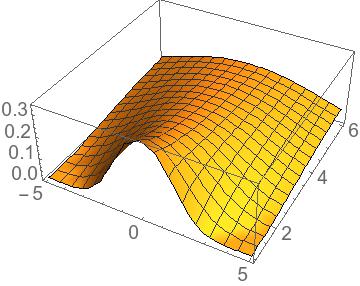}}
\subfloat[$\beta=3/2$.]{
\includegraphics[scale=0.4]{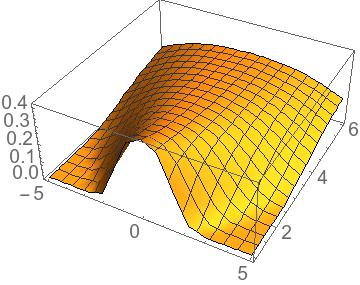}}
\subfloat[$\beta=15/8$.]{
\includegraphics[scale=0.4]{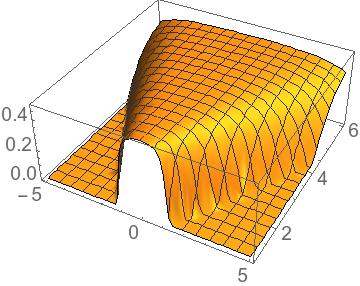}}
\caption{Graphs of the Green functions $Y_{2,\beta,0,1}(t,x)$ for $1<\beta<2$ for $1\le t\le 6$ and $|x|\le 5$.}
\label{F:Ytx}
\end{figure}

\subsection{Proof of Theorem \ref{T:PDE}}\label{SS:PDE}

\begin{proof}[Proof of Theorem \ref{T:PDE}]
Equations \eqref{E:Zab}--\eqref{E:FZ*} have been proved in \cite{CHHH15Time} when $\gamma=0$.
Let $\widehat{f}$ and $\widetilde{g}$ denote the Fourier transform in the
space variable and the Laplace transform in the
time variable, respectively.
Apply the Fourier transform to \eqref{E:PDE} to obtain
\[
\begin{cases}
\displaystyle \partial^\beta \widehat{u}(t,\xi)+\frac{\nu}{2}|\xi|^\alpha
\widehat{u}(t,\xi)=I_t^\gamma \left[\widehat{f}(t,\xi)\right]\;,& \xi\in\R^d\\[0.5em]
\displaystyle
\left.\frac{\partial^k}{\partial t^k} \widehat{u}(t,\xi)\right|_{t=0} = \widehat{u}_k(\xi)\;,
&\text{$0\le k\le \Ceil{\beta}-1$, $\xi\in\R^d$\:.}
\end{cases}
\]
Apply the Laplace transform on the Caputo derivative using \cite[Theorem 7.1]{Die04}:
\[
\calL\left[\partial^\beta\: \widehat{u}(t,\xi)\right](s) = s^\beta \;\widetilde{\widehat{u}}(s,\xi) -
\sum_{k=0}^{\Ceil{\beta}-1} s^{\beta-1-k}\; \widehat{u}_k(\xi).
\]
On the other hand, it is known that (see, e.g., \cite[(7.14)]{SamkoKilbasMarichev93}),
\[
\calL I_t^\gamma\left[\widehat{f}(t,\xi)\right] = s^{-\gamma} \widetilde{\widehat{f}}(s,\xi),\quad \Re(\gamma)>0.
\]
Thus,
\[
\widetilde{\widehat{u}}(s,\xi) = \left(s^\beta
+\frac\nu2 \:|\xi|^\alpha \right)^{-1}\left[\sum_{k=0}^{\Ceil{\beta}-1} s^{\beta-1-k}\; \widehat{u}_k(\xi)
+s^{-\gamma} \widetilde{\widehat{f}}(s,\xi)\right].
\]
Notice that (see \cite[(1.80)]{Podlubny99FDE})
\[
\calL\left[t^{\beta-1}E_{\alpha,\beta}(-\lambda t^\alpha) \right](s) = \frac{s^{\alpha-\beta}}{s^\alpha+\lambda},
\quad\text{for $\Re(s)>|\lambda|^{1/\alpha}$.}
\]
Hence,
\[
\widehat{u}(t,\xi)=\sum_{k=0}^{\Ceil{\beta}-1}t^k E_{\beta,k+1}\left(-\frac{\nu}{2}|\xi|^\alpha t^\beta\right)\widehat{u}_k(\xi)
+
\int_0^t \ud \tau
\: \tau^{\beta+\gamma-1} E_{\beta,\beta+\gamma}\left(-\frac{\nu}{2}|\xi|^\alpha\tau^\beta\right)\widehat{f}(\tau,\xi),
\]
from which \eqref{E:FZ}--\eqref{E:FZ*} are proved.
The expressions for $Z$ and $Z^*$ in \eqref{E:Zab} and \eqref{E:Z*ab}, respectively,  are proved in \cite{CHHH15Time}.
By the fact that (see \cite[(1.82)]{Podlubny99FDE})
\[
\lMr{t}{D}{+}^\gamma \left(t^{\beta-1}E_{\alpha,\beta}(\lambda t^\alpha)\right) = t^{\beta-\gamma-1}
E_{\alpha,\beta-\gamma}(\lambda t^\alpha), \quad \gamma\in\R.
\]
Recall that $\DRL{+}{\alpha}$ is the Riemann-Liouville fractional derivative of order $\alpha\in\R$ (see \eqref{E:RLD}).
Hence,
we see that
\[
Y(t,x)= \lMr{t}{D}{+}^{\theta} Z(t,x), \quad \text{with $\theta:=\Ceil{\beta}-\beta-\gamma$},
\]
which can be evaluated using \cite[Theorem 2.8]{KilbasSaigo04H} in the same way as in \cite{CHHH15Time} for the case $\gamma=0$.
%
This completes the proof of Theorem \ref{T:PDE}.
\end{proof}

\subsection{Nonnegativity of the fundamental solutions (proof of Theorem \ref{T:NonY})}
\label{SS:NonY}

We first prove some lemmas.

\begin{lemma}\label{L_:NonNeg}
The following Fox H-functions are nonnegative:
\begin{enumerate}[(1)]
 \item for all $\theta\in (0,1)$,
 \begin{align}\label{E_:N1}
\FoxH{1,1}{2,2}{x}
{(0,1),\:\left(0,\theta\right)}
{(0,1),\:\left(0,\theta\right)}
=\frac{1}{\pi} \frac{x^{1/\theta}}{1+2 x^{1/\theta}\cos(\pi\theta)+x^{2/\theta}}
> 0,\quad\text{for $x>0$};
\end{align}
\item for all $\mu>0$ and $0<\theta\le \min(1,\mu)$,
\begin{align}\label{E_:H1}
 \RR\ni x \mapsto \FoxH{1,0}{1,1}{|x|}{(\mu,\theta)}{(1,1)}\ge 0.
\end{align}
\item for all $d\in\bbN$ and $\alpha\in (0,2]$,
\begin{align}\label{E_:NonStable}
\R\ni x\mapsto \FoxH{1,1}{1,2}{|x|}{(1,1)}{(d/2,\alpha/2),\:(1,\alpha/2)}> 0.
\end{align}
\end{enumerate}
\end{lemma}
\begin{proof}
(2) and (3) are covered by Lemma 4.5 and Theorem 3.3 of \cite{CHHH15Time}, respectively.
As for (1),
expression \eqref{E_:N1} can be found in \cite[(4.38)]{MainardiEtc01Fundamental} for
the {\it neutral-fractional diffusions}.
For completeness, we give a proof here.
Because the parameters $\Delta$ and $\delta$, defined in \eqref{E:Delta} and \eqref{E:delta}, of this Fox H-function are equal to $0$ and $1$, respectively,
Theorem 1.3 implies that for $x\in(0,1)$,
\begin{align}
\notag
\FoxH{1,1}{2,2}{x}
{(0,1),\:\left(0,\theta\right)}
{(0,1),\:\left(0,\theta\right)}
&=\sum_{k=0}^\infty \frac{(-1)^k}{k!} \frac{\Gamma(k+1)}{\Gamma(-k\theta)\Gamma(1+k\theta)} x^{k/\theta}\\
\label{E:sin}
&=
\sum_{k=0}^\infty -(-1)^k \sin(\pi k\theta) x^{k/\theta}\\ \notag
&=\Im\left(\sum_{k=0}^\infty -(-1)^k e^{k\theta i} x^{k/\theta}\right)\\ \notag
&=-\Im \frac{1}{1+e^{\theta i} x^{1/\theta}} \\ \notag
&=\frac{1}{\pi} \frac{x^{1/\theta}}{1+2 x^{1/\theta}\cos(\pi\theta)+x^{2/\theta}},
\end{align}
where we have applied \cite[(5.5.3)]{NIST2010} in \eqref{E:sin}. Similarly, when $x>1$, Theorem 1.4
of \cite{KilbasEtc06} implies that
\begin{align*}
\FoxH{1,1}{2,2}{x}
{(0,1),\:\left(0,\theta\right)}
{(0,1),\:\left(0,\theta\right)}
&=\sum_{k=0}^\infty \frac{(-1)^k}{k!} \frac{\Gamma(k+1)}{\Gamma(-(1+k)\theta)\Gamma(1+(1+k)\theta)} x^{-(1+k)/\theta}\\
&=
\sum_{k=0}^\infty (-1)^k \sin(\pi (1+k)\theta) x^{-(1+k)/\theta}\\
&=\Im\left(\sum_{k=0}^\infty -(-1)^k e^{k\theta i} x^{-k/\theta}\right)\\
&=-\Im \frac{1}{1+e^{\theta i} x^{-1/\theta}} \\
&=\frac{1}{\pi} \frac{x^{1/\theta}}{1+2 x^{1/\theta}\cos(\pi\theta)+x^{2/\theta}}.
\end{align*}
Finally, the existence of this Fox H-function at $x=1$ is not covered by Theorem 1.1 of \cite{KilbasSaigo04H}
because $\Delta=0$ and $\mu=0$ (see \eqref{E:mu} for the definition of the parameter $\mu$).
In fact, as one can see that the series in \eqref{E:sin} with $x=1$ diverges.
Nevertheless, we may define that
\[
\FoxH{1,1}{2,2}{1}
{(0,1),\:\left(0,\theta\right)}
{(0,1),\:\left(0,\theta\right)}
:= \lim_{x\rightarrow 1 }\frac{1}{\pi} \frac{x^{1/\theta}}{1+2 x^{1/\theta}\cos(\pi\theta)+x^{2/\theta}}
=\frac{1}{2\pi} \frac{1}{1+\cos(\pi\theta)}>0.
\]
This completes the proof of Lemma \ref{L_:NonNeg}.
\end{proof}

\begin{lemma}\label{L:f}
For $\mu\in (0,2]$, $\theta\in (0,2]$ and $d\ge 1$, the function
\[
f_{d,\mu,\theta}(x) := x^{-d} \FoxH{2,0}{1,2}{x^2}{(\mu,\theta)}{(d/2,1),\:(1,1)}, \quad x>0,
\]
has the following properties:
\begin{enumerate}[(a)]
 \item $\displaystyle \frac{\ud}{\ud x}f_{d,\mu,\theta}(x) = - 2 x f_{d+2,\mu,\theta}(x) $.
 \item $\displaystyle f_{d,\mu,\theta}(x) = \frac{2}{\sqrt{\pi}}\int_x^\infty \ud z \frac{z}{\sqrt{z^2-x^2}} f_{d+1,\mu,\theta}(z)$ for $x>0$.
 \item $f_{d,\mu,\theta}(x)\ge 0$ for all $x>0$ if $\theta\le 2 \min(1,\mu)$ and $d\le 3$.
\end{enumerate}
\end{lemma}
\begin{proof}
{\bf (a)~} Apply \cite[Property 2.8]{KilbasSaigo04H} with $k=1$, $w=-d$, $c=1$ and $\sigma=2$ to get
\[
\frac{\ud}{\ud x}f_{d,\mu,\theta}(x) = x^{-d-1}
\FoxH{2,1}{2,3}{x^2}{(d,2),\:(\mu,\theta)}{(d/2,1),\:(1,1),\:(1+d,2)}.
\]
By the recurrence relation of the Gamma function, we see that
\[
\frac{\Gamma(1-d-2s)}{\Gamma(-d-2s)}\: \Gamma(d/2+s)
=-2 (s+d/2) \Gamma(d/2+s)=-2 \Gamma(1+d/2+s).
\]
By the definition of the Fox H-function, the above expression can be simplified as
\[
\frac{\ud}{\ud x}f_{d,\mu,\theta}(x) = -2 x^{-d-1}
\FoxH{2,1}{1,2}{x^2}{(\mu,\theta)}{((d+2)/2,1),\:(1,1)}
=-2 x f_{d+2,\mu,\theta}(x).
\]

{\vspace{1em}\noindent\bf (b)~}
By the definition of the Fox H-function,
\begin{align}\label{E_:fd+1}
f_{d+1,\mu,\theta}(x)=x^{-(d+1)}\frac{1}{2\pi i}\int_{L_{i\gamma\infty}} \frac{\Gamma((d+1)/2+s)\Gamma(1+s)}{\Gamma(1-\mu-\theta s)} x^{-2s} \ud s,\quad
\text{for any $\gamma>-1$},
\end{align}
where the contour $L_{i\gamma\infty}$ is defined in Definition \ref{D:H}.
Assuming that we can switch the integrals,
which can be made rigorous by writing $f$ in the series form and applying Fubini's theorem,
we see that
\[
\int_x^\infty \ud z \: \frac{z}{\sqrt{z^2-x^2}}f_{d+1,\mu,\theta}(z)=
\frac{1}{2\pi i}\int_{L_{i\gamma\infty}}\ud s
\frac{\Gamma((d+1)/2+s)\Gamma(1+s)}{\Gamma(1-\mu-\theta s)}
\int_x^\infty \ud z\: \frac{z^{-2s-d}}{\sqrt{z^2-x^2}}.
\]
By change of variable $(z/x)^2-1=y$ and Euler's Beta integral (see, e.g., \cite[5.12.3 on p.142]{NIST2010}),
we see that
\begin{align*}
 \int_x^\infty \ud z\: \frac{z^{-2s-d}}{\sqrt{z^2-x^2}}
 =\frac{x^{-2s-d}}{2}\int_0^\infty y^{\frac{1}{2}-1} (1+y)^{-\frac{1}{2}-\frac{2s+d}{2}}\ud y=
 \frac{x^{-2s-d}}{2} \frac{\sqrt{\pi}\: \Gamma(d/2+s)}{\Gamma((d+1)/2+s)}.
\end{align*}
Note that the above integral is convergent provided that $\Re(2s+d)>0$, which is satisfied by choosing, e.g.,
$\gamma=\Re(s)=0$ in \eqref{E_:fd+1}.
Therefore,
\[
\int_x^\infty \ud z \: \frac{z}{\sqrt{z^2-x^2}}f_{d+1,\mu,\theta}(z)=
\frac{\sqrt{\pi}}{2}  x^{-d}\frac{1}{2\pi i}\int_{L_{i\gamma\infty}} \frac{\Gamma(d/2+s)\Gamma(1+s)}{\Gamma(1-\mu-\theta s)} x^{-2s} \ud s
=
\frac{\sqrt{\pi}}{2} f_{d,\mu,\theta}(x).
\]

{\vspace{1em}\noindent\bf (c)~}
By the recurrence in (b), we only need to prove the case $d=3$.
Apply Lemma \ref{L:H1/2}, Properties 2.4 and 2.5 in \cite{KilbasEtc06} to obtain
\begin{align*}
f_{3,\mu,\theta}(x)&=\frac{\sqrt{\pi}}{4}x^{-3}\FoxH{1,0}{1,1}{4x^2}{(\mu,\theta)}{(2,2)}\\
&=\frac{\sqrt{\pi}}{8}x^{-3}\FoxH{1,0}{1,1}{2x}{(\mu,\theta/2)}{(2,1)}\\
&=\frac{\sqrt{\pi}}{16}x^{-4}\FoxH{1,0}{1,1}{2x}{(\mu-\theta/2,\theta/2)}{(1,1)}.
\end{align*}
Then (c) is proved by an application of part (2) of Lemma \ref{L_:NonNeg}.
%
\end{proof}

\bigskip
\begin{proof}[Proof of Theorem \ref{T:NonY}]
By comparing the Fox H-functions in \eqref{E:Zab}, \eqref{E:Yab}, and \eqref{E:Z*ab},
We only need to consider the following Fox H-function:
\[
g(x)=
\FoxH{2,1}{2,3}{x}{(1,1),\: (\eta,\beta)}
{(d/2,\alpha/2),\: (1,1),\: (1,\alpha/2) },\quad x>0.
\]
The parameter $\eta$ takes the following values
\[
\eta =
\begin{cases}
 \Ceil{\beta}& \text{in case of $Z$},\\
 \beta+\gamma& \text{in case of $Y$},\\
 1& \text{in case of $Z^*$}.
\end{cases}
\]

{\noindent\bf (1)~}
If $\beta=1$ and $\gamma=0$, then $Z=Y$ and by Property 2.2 of \cite{KilbasEtc06},
\[
g(x) = \FoxH{2,0}{1,2}{x}{(1,1)}{(d/2,\alpha/2),\:(1,\alpha/2)}, \qquad x>0
\]
which is positive by part (3) of Lemma \ref{L_:NonNeg}.
If $\beta<1$, then we can apply Theorem \ref{T:HConvH} to obtain that
\begin{align}\label{E:HHH2}
g(x)=\int_0^\infty
\FoxH{1,1}{1,2}{t}{(1,1)}{(d/2,\alpha/2),\:(1,\alpha/2)}
\FoxH{1,0}{1,1}{\frac{x}{t}}{(\eta,\beta)}{(1,1)}\frac{\ud t}{t}.
\end{align}
In fact, conditions \eqref{E:HConvH} are satisfied because
\[
A_1=0,\quad B_1=d/\alpha, \quad A_2=1, \quad B_2=\infty.
\]
Moreover, $a_1^*=1$ and $\beta\in(0,1)$ implies that $a_2^*=1-\beta>0$.
Hence, condition (1) of Theorem \ref{T:HConvH} is satisfied.
This proves \eqref{E:HHH2}.
If $\beta=1$ and $\gamma>0$, then $a_2^*=\Delta_2=0$. In view of condition (3) of Theorem \ref{T:HConvH},
relation \eqref{E:HHH2} is still true if
$\Re(\mu_2)>-1$ with $\mu_2=1-\eta$, i.e., $\gamma>1$.
The two Fox H-functions in \eqref{E:HHH2} are nonnegative by parts (2) and (3) of Lemma \ref{L_:NonNeg}.

{\bigskip\noindent\bf (2)~}
In this case, we have that $d\le 3$.
When $\alpha=2$, by Property 2.2 of \cite{KilbasEtc06} and Lemma \ref{L:f},
\[
g(x)=\FoxH{2,0}{1,2}{x}{(\eta,\beta)}{(d/2,1),\:(1,1)}= x^{d/2}f_{d,\eta,\beta}(\sqrt{x})\ge 0,\quad x>0,
\]
because $\beta<2\le 2\min(1,\eta)$.
If $\alpha\ne 2$, by Property 2.2 of \cite{KilbasEtc06}, we see that
\[
g(x)=
\FoxH{3,1}{3,4}{x}{(1,1),\: (\eta,\beta),\: (1,\alpha/2) }
{(d/2,\alpha/2),\: (1,\alpha/2),\: (1,1),\: (1,\alpha/2) },\quad x>0,
\]
As in the previous case, by Theorem \ref{T:HConvH}, we see that
\begin{align}\label{E:HHH}
\int_0^\infty
\FoxH{2,0}{1,2}{t}{(\eta,\beta)}{(d/2,\alpha/2),\:(1,\alpha/2)}
\FoxH{1,1}{2,2}{\frac{x}{t}}{(1,1),\:(1,\alpha/2)}{(1,1),\:(1,\alpha/2)}\frac{\ud t}{t} =
\FoxH{3,1}{3,4}{x}{(1,1),\: (\eta,\beta),\: (1,\alpha/2) }
{(d/2,\alpha/2),\: (1,\alpha/2),\: (1,1),\: (1,\alpha/2) }.
\end{align}
Note that condition \eqref{E:HConvH} is satisfied because in this case,
\[
A_1=\infty,\quad B_1=\min(d,2)/\alpha,\quad
A_2=1,\quad B_2=0.
\]
When $\alpha<\beta$, then
\[
a_1^*=\alpha-\beta>0,\quad a_2^*=2-\alpha>0,
\]
and condition (1) in Theorem \ref{T:HConvH} is satisfied.
When $1<\beta=\alpha<2$, then
\[
a_2^*=2-\alpha>0,\quad a_1^*=\Delta_1=0,\quad \Re(\mu_1)=1+\frac{d}{2}-\eta-\frac{1}{2}.
\]
Hence, in view of condition (2) of Theorem \ref{T:HConvH},
the integral \eqref{E:HHH} is still true if $1+d/2-\eta-1/2<-1$, i.e., $\gamma>(d+3)/2-\beta$.

Now, by Property 2.4 of \cite{KilbasEtc06}, the first Fox H-function in \eqref{E:HHH} is equal to
\[
\frac{2}{\alpha}\FoxH{2,0}{1,2}{t^{2/\alpha}}{(\eta,2\beta/\alpha)}{(d/2,1),(1,1)} =
\frac{2}{\alpha}t^{d/\alpha} f_{d,\eta,2\beta/\alpha}(t^{1/\alpha}).
\]
By Lemma \ref{L:f} (c), we see that under the condition that
$\frac{2\beta}{\alpha}\le 2\min(1,\eta)$,
the first Fox H-function in \eqref{E:HHH} is nonnegative.
This condition is satisfied if $1\le \beta\le \alpha\le 2$.
By Property 2.3 in \cite{KilbasSaigo04H}, the second Fox H-function in \eqref{E:HHH}
is equal to
\[
\FoxH{1,1}{2,2}{\frac{t}{x}}{(0,1),\:(0,\alpha/2)}{(0,1),\:(0,\alpha/2)}.
\]
Thanks to Lemma \ref{L_:NonNeg} (1), this function is strictly positive for $t/x\ne 0$ when $\alpha\in (0,2)$.

\bigskip{\noindent\bf (3)~}
Now we consider the case when $d\ge 4$.
The case $\alpha=2$ is covered by Lemma 25 of \cite{Pskhu09}.
In the following, we assume that $\alpha \in (0,2)$.
By the scaling property, we may only consider the case $t=1$.
Hence, it suffices to study the following function
\[
g(x)= x^{-d} \FoxH{2,1}{2,3}{x^\alpha}{(1,1),\:(\beta,\beta)}{(d/2,\alpha/2),\:(1,1),\:(1,\alpha/2)}, \quad x>0.
\]
Because $a^*=2-\beta>0$, we can apply Theorem 1.7 of \cite{KilbasEtc06} to obtain that
\[
g(x) = -\frac{\Gamma((d+\alpha)/2)}{\Gamma(2\beta)\Gamma(-\alpha/2)} x^{-d-1} + O(x^{-(d+1)}),\quad \text{as $x\rightarrow\infty$.}
\]
The condition $\alpha\in (0,2)$ implies that $\Gamma(-\alpha/2)<0$. Thus, the coefficient of $x^{-d-1}$ is positive.
Hence, $g$ can assume positive values.

Now we consider the behavior of $g(x)$ around zero.
Because $\beta>1$ and $2\alpha< 4\le d$, we can apply the case 6 of Lemma \ref{L:HAt0}:
\[
g(x) = - \frac{\Gamma((d-2\alpha)/2)}{\Gamma(-\beta)\Gamma(\alpha)} x^{2\alpha-d} + O(x^{\min(3\alpha-d,0)}),
\quad\text{as $x\rightarrow 0_+$.}
\]
The coefficient of $x^{2\alpha-d}$ is negative because $\Gamma(-\beta)>0$ for $\beta\in(1,2)$.
Therefore, $g(x)$ can assume negative values.
This completes the proof of Theorem \ref{T:NonY}.
\end{proof}

\section{Proofs of Theorems \ref{T2:ExUni} and \ref{T:ExUni}}\label{S:Proof}

The proofs of Theorems \ref{T2:ExUni} and \ref{T:ExUni} will follow the
same arguments as the proof of \cite[Theorem 1.2]{ChenDalang13Heat},
which requires some lemmas and propositions.

\subsection{Dalang's condition}
\begin{lemma}\label{L:ML-Est}
Suppose that $\theta>1/2$ and $\beta\in (0,2)$. The following statements are true:
\begin{enumerate}[(a)]
 \item There is some nonnegative constant $C_{\beta,\theta}$ such that for all $t>0$ and $\lambda>0$,
 \begin{align}\label{E:ML-Est1}
\int_0^t w^{2(\theta-1)} E^2_{\beta,\theta}(-\lambda w^\beta )\ud w \le
C_{\beta,\theta}\:\frac{t^{2\theta-1}}{1+(t\lambda^{1/\beta})^{\min(2\beta,2\theta-1)}},
\end{align}
\item If $\beta\le \min(1,\theta)$, then for some nonnegative constant $C_{\beta,\theta}'$,
 \[
\int_0^t w^{2(\theta-1)} E^2_{\beta,\theta}(-\lambda w^\beta )\ud w \ge
C_{\beta,\theta}'\: \frac{t^{2\theta-1}}{1+(t\lambda^{1/\beta})^{\min(2\beta,2\theta-1)}},
\]
for all $t>0$ and $\lambda>0$.
\end{enumerate}
\end{lemma}
\begin{remark}
When $\theta=\beta=1$, then $E_1(x)=e^x$ and thus $\int_0^t e^{-2\lambda w}\ud w = (2\lambda)^{-1}(1-e^{-2t\lambda})$ and \eqref{E:ML-Est1} is clear for this case.
\end{remark}

\begin{proof}[Proof of Lemma \ref{L:ML-Est}]
(a) In this case, by the asymptotic property of the Mittag-Leffler function (see \cite[Theorem 1.6]{Podlubny99FDE}),
for some nonnegative constants $C_i$'s,
\begin{align}\label{E_:WEC}
\int_0^t w^{2(\theta-1)} E^2_{\beta,\theta}(-\lambda w^\beta )\ud w
&\le
C_1 \int_0^t  \frac{w^{2(\theta-1)}}{(1+\lambda^{1/\beta} w)^{2\beta}}\ud w \\ \notag
&=C_2 \: t^{2\theta-1}\int_0^1  \frac{u^{2(\theta-1)}}{(1+\lambda^{1/\beta} t  u)^{2\beta}}\ud u \\ \label{E_:WEC2}
&=C_3 \:t^{2\theta-1}\lMr{2}{F}{1}(2\beta,2\theta-1,2\theta;-t\lambda^{1/\beta})\\ \label{E_:WEC3}
&=C_4 \:t^{2\theta-1}
\FoxH{1,2}{2,2}{t\lambda^{1/\beta}}{(1-2\beta,1),\:(2(1-\theta),1)}{(0,1),\:(1-2\theta,1)},
\end{align}
where in \eqref{E_:WEC2} we have applied \cite[15.6.1]{NIST2010} under the condition that $\theta>1/2$, and \eqref{E_:WEC3} is due to \cite[(2.9.15)]{KilbasSaigo04H}.
Notice that $\Delta=0$ for the above Fox H-function, which allows us to apply Theorems 1.7 and 1.11 of \cite{KilbasSaigo04H}.
In particular, by \cite[Theorem 1.11]{KilbasSaigo04H}, we know that
\[
\FoxH{1,2}{2,2}{x}{(1-2\beta,1),\:(2(1-\theta),1)}{(0,1),\:(1-2\theta,1)}\sim O(1),\quad\text{as $x\rightarrow 0$}.
\]
When $\theta\ne \beta$, by \cite[Theorem 1.7]{KilbasSaigo04H},
\[
\FoxH{1,2}{2,2}{x}{(1-2\beta,1),\:(2(1-\theta),1)}{(0,1),\:(1-2\theta,1)}
\sim O\left(x^{-\min(2\beta,2\theta-1)}\right),\quad\text{as $x\rightarrow \infty$}.
\]
In particular, when $\theta=\beta$, by Property 2.2 and (2.9.5) of \cite{KilbasSaigo04H},
\[
\FoxH{1,2}{2,2}{x}{(1-2\beta,1),\:(2(1-\theta),1)}{(0,1),\:(1-2\theta,1)}
=
\FoxH{1,1}{1,1}{x}{(2(1-\theta),1)}{(0,1)}
=
\Gamma(2\theta-1)(1+x)^{1-2\theta}.
\]
(b) When $\beta<\min(1,\theta)$, by \eqref{E:E-CM}, we know that $E_{\beta,\theta}(-|x|)$ is nonnegative, hence, for another nonnegative
constant $C'$, one can reverse the inequality \eqref{E_:WEC}. This completes the proof of Lemma \ref{L:ML-Est}.
\end{proof}

\begin{lemma}[Dalang's condition]\label{L:Dalang}
Let $Y(t,x)=Y_{\alpha,\beta,\gamma,d}(t,x)$ with $\alpha\in (0,2]$, $\beta\in (0,2)$, and $\gamma>0$.
The following two conditions are equivalent:
\[
\text{(i)~~}d<2\alpha+\frac{\alpha}{\beta}\min(2\gamma-1,0)\qquad
\Longleftrightarrow\qquad \text{(ii)~~}\int_0^t\ud s\int_{\R^d} \ud y \: Y(s,y)^2<\infty,\quad\text{for all $t>0$}.
\]
\end{lemma}
\begin{proof}
(i)$\Rightarrow$(ii):~
Fix an arbitrary $t>0$. By the Plancherel theorem and \eqref{E:FY}, we only need to prove that
\begin{align}\label{E:DalangPlan}
\int_0^t\ud s \int_{\R^d}\ud \xi\: s^{2(\beta+\gamma-1)}E_{\beta,\beta+\gamma}^2(-s^\beta|\xi|^\alpha)<+\infty.
\end{align}
Notice that $d>0$ and Condition (1) together imply that $\beta+\gamma>1/2$. Thus, we can integrate $\ud s$ first using Lemma \ref{L:ML-Est} (a). Then it reduces to prove that
\[
\int_{\R^d} \frac{1}{1+|\xi|^{2\alpha+\frac{\alpha}{\beta}\min(2\gamma-1,0) }}\ud \xi<+\infty,
\]
which is guaranteed by (i).

\bigskip\noindent (ii)$\Rightarrow$(i):~
The case $\beta\in(0,1]$ can be proved in the same way as above by an application of Lemma \ref{L:ML-Est} (b).
The case $\beta\in(1,2)$ is trickier.
Fix $t>0$. Denote the integral in \eqref{E:DalangPlan} by $I(t)$. Then by change of variables,
\[
I(t)= C\int_0^t\ud s \: s^{2(\beta+\gamma-1)-\frac{\beta d}{\alpha}} \int_0^\infty \ud y\: y^{\frac{d}{\alpha}-1}
E_{\beta,\beta+\gamma}^2(-y).
\]
Note that the double integral is decoupled. The integrability of $\ud s$ at zero implies
that $2(\beta+\gamma)-1 -\frac{\beta d}{\alpha}>0$, which is equivalent to
\begin{align}\label{E_:D1}
d<2\alpha +\frac{\alpha}{\beta}(2\gamma-1).
\end{align}
By the asymptotics of $\E_{\beta,\beta+\gamma}(-y)$ at $+\infty$ (see, e.g., Theorem 1.3 in \cite{Podlubny99FDE}; note
that the condition $\beta\in(1,2)$ is used here),
we see that there exist $y_0>0$ and some constant $C>0$ such that
\[
E_{\beta,\beta+\gamma}^2(-y)\ge \frac{C}{1+y^2}\qquad\text{for all $y\ge y_0$.}
\]
Hence, the integrability of $\ud y$ at zero and infinity implies the following conditions:
\begin{align}\label{E_:D2}
d/\alpha>0 \qquad\text{and}\qquad d/\alpha-3<1.
\end{align}
Combining \eqref{E_:D1} and \eqref{E_:D2} gives (i).
This completes the proof of Lemma \ref{L:Dalang}.
\end{proof}

\subsection{Some continuity results on $Y$}

This part contains some continuity results on $Y$.
All the results proved in this part will be used in the proof of Theorem \ref{T:ExUni}.
In particular, Proposition \ref{P:G-SD} will be used to prove the H\"older continuity (Theorem \ref{T:Holder}).

\begin{proposition}\label{P:G-SD}
Suppose $\alpha\in (0,2]$, $\beta\in(0,2)$, $\gamma\ge 0$, and \eqref{E:Dalang} holds.
Then $Y(t,x)=Y_{\alpha,\beta,\gamma,d}(t,x)$ satisfies the following two properties:
\begin{enumerate}[(i)]
 \item
 For all $0<\theta<(\Theta-d)\wedge 2$ and $T>0$,
 there is some nonnegative constant $C=C(\alpha,\beta,\gamma,\nu,\theta,T,d)$ such that
 for all $t\in (0,T]$ and $x,y\in\R^d$,
\begin{align}\label{E:G-x}
\iint_{\R_+\times\R^d}\ud r\ud z\: \left(
Y\left(t-r,x-z\right)
-Y\left(t-r,y-z\right)
\right)^2\le C \: |x-y|^{\theta},
\end{align}
\item If $\beta\le 1$ and $\gamma\le \Ceil{\beta}-\beta$, then there is some nonnegative constant $C=C(\alpha,\beta,\gamma,\nu,d)$ such that
for all $s,t\in\R_+^*$ with $s\le t$, and $x\in\R^d$,
\begin{align}\label{E:G-t1}
\int_0^s \ud r\int_{\R^d}\ud z \left(Y\left(t-r,x-z\right)
-Y\left(s-r,x-z\right)
\right)^2\le C (t-s)^{2(\beta+\gamma)-1-d\beta/\alpha}\;,
\end{align}
and
\begin{align}\label{E:G-t2}
\int_s^t\ud r\int_{\R^d}\ud z \;
Y^2\left(t-r,x-z\right)
\le C (t-s)^{2(\beta+\gamma)-1-d\beta/\alpha}\;.
\end{align}
\end{enumerate}
\end{proposition}
\begin{proof}
(i) Fix $t>0$. By Plancherel's theorem and \eqref{E:FY},
the left hand side of \eqref{E:G-x} is equal to
\begin{align*}
&\frac{1}{(2\pi)^d}\int_0^t \ud r\:(t-r)^{2(\beta+\gamma-1)}\int_{\R^d} \ud \xi \:
E_{\beta,\beta+\gamma}^2
\left(-2^{-1}\nu(t-r)^\beta |\xi|^\alpha\right)\;
\left|e^{-i\xi\cdot x}-e^{-i\xi\cdot y}\right|^2\\
=&\frac{1}{(2\pi)^d}\int_{\R^d} \ud \xi\:
\left(1-\cos(\xi\cdot (x-y))\right) \int_0^t\ud r\:(t-r)^{2(\beta+\gamma-1)}
\;E_{\beta,\beta+\gamma}^2\left(-2^{-1}\nu(t-r)^\beta |\xi|^\alpha\right)\\
\le&
\frac{1}{(2\pi)^d}\int_{\R^d} \ud \xi\:
\left(1-\cos(\xi\cdot (x-y))\right)
\frac{C_{\beta,\gamma,T}}{1+|\xi|^{2\alpha+\frac{\alpha}{\beta}\min(0,2\gamma-1)}}
\end{align*}
where we have applied Lemma \ref{L:ML-Est} in the last step (see also the proof of Lemma \ref{L:Dalang}).
Denote $\Theta:=2\alpha+\frac{\alpha}{\beta}\min(0,2\gamma-1)$.
Because $1-\cos(x)\le 2\wedge \left(x^2/2\right)$ for all $x\in\R$,
we only need to bound
\begin{align*}
\int_{\R^d}\ud \xi\:\frac{2\wedge
\left[|x-y|\: |\xi| \right/\sqrt{2}\:]^2}{1+|\xi|^{\Theta}}
\le C'
\Bigg(&\quad
|x-y|^{-d}\int_0^{\sqrt{2}}\frac{u^{d+1}}{(1+|x-y|^{-1}u)^\Theta}\ud u\\
&+
|x-y|^{\Theta-d}
\int_{\sqrt{2}}^\infty\frac{2}{u^{\Theta+1-d}}\ud u
\Bigg)
\end{align*}
The second integral on the right hand side of the above inequality is finite
provided that $\Theta>d$, which is Dalang's condition.
By \cite[15.6.1]{NIST2010}, for some constant $C>0$,
\[
\int_0^{\sqrt{2}}\frac{u^{d+1}}{(1+|x-y|^{-1}u)^\Theta}\ud u
=C\lMr{2}{F}{1}(\Theta,2+d,3+d;-\sqrt{2} |x-y|^{-1}),
\]
which is true under the condition that $d+3>d+2>0$.
%
By \cite[2.9.15]{KilbasSaigo04H},
\[
\int_0^{\sqrt{2}}\frac{u^{d+1}}{(1+|x-y|^{-1}u)^\Theta}\ud u
=C''
\FoxH{1,2}{2,2}{\sqrt{2}|x-y|^{-1}}{(-1-d,1),\:(1-\Theta,1)}{(0,1),\: (-2-d,1)}.
\]
Since $\Delta=0$, by \cite[Theorem 1.7]{KilbasSaigo04H}, for all $\theta\in (0,\min(\Theta,2+d))$,
\[
\int_0^{\sqrt{2}}\frac{u^{d+1}}{(1+|x-y|^{-1}u)^\Theta}\ud u = O(|x-y|^{\theta-d}), \quad \text{as $|x-y|\rightarrow 0$.}
\]
Combining these cases, we have proved (i).
\bigskip

\noindent(ii) Denote the left hand side of \eqref{E:G-t1} by $I$.
Apply Plancherel's theorem and use \eqref{E:E-CM},
\begin{align*}
  I=C \int_0^s\ud r\: \int_{\R^d} \ud \xi\:  \Big|\quad&
 (t-r)^{\beta+\nu-1} E_{\beta,\beta+\nu}\left(-2^{-1}\nu(t-r)^\beta |\xi|^\alpha \right)\\
 -&(s-r)^{\beta+\nu-1}E_{\beta,\beta+\nu}\left(-2^{-1}\nu(s-r)^\beta |\xi|^\alpha\right)
\Big|^2.
\end{align*}
Then by Lemma \ref{L:Y^2} below,
\[
I= C C_\sharp\int_0^s \ud r \left[(t-r)^{2(\beta+\gamma-1)-d\beta/\alpha}+(s-r)^{2(\beta+\gamma-1)-d\beta/\alpha}\right]
- 2C \int_0^s\ud r [(t-r)(s-r)]^{ \beta+\gamma-1 } \: H(r),
\]
where
\[
H(r)=\int_{\R^d}
E_{\beta,\beta+\gamma}(-2^{-1}\nu(t-r)^\beta|\xi|^\alpha)
E_{\beta,\beta+\gamma}(-2^{-1}\nu(s-r)^\beta|\xi|^\alpha)\ud \xi.
\]
By \eqref{E:E-CM} and $t\ge s$,
\begin{align*}
 H(r)&\ge\int_{\R^d}
E_{\beta,\beta+\gamma}^2(-2^{-1}\nu(t-r)^\beta|\xi|^\alpha)
\ud \xi\\
&=
(t-r)^{-2(\beta+\gamma-1)}
\int_{\R^d} (t-r)^{2(\beta+\gamma-1)}
E_{\beta,\beta+\gamma}^2(-2^{-1}\nu(t-r)^\beta|\xi|^\alpha)
\ud \xi\\
&=
(t-r)^{-2(\beta+\gamma-1)} \int_{\R^d} Y(t-r,y)^2 \ud y\\
&= C_\sharp (t-r)^{-d\beta/\alpha},
\end{align*}
where in the last step we have applied Lemma \ref{L:Y^2}.
Because  $\beta+\gamma\le 1$, we see that
\begin{align*}
\int_0^s\ud r [(t-r)(s-r)]^{ \beta+\gamma-1 } \: H(r)
  \ge C_{\sharp}\int_0^s \ud r  (t-r)^{2(\beta+\gamma)-1-d\beta/\alpha}.
\end{align*}
Denote $\rho:=2(\beta+\gamma)-1-d\beta/\alpha$. Note that $\rho>0$ is implied by Dalang's condition \eqref{E:Dalang}.
Therefore,
\begin{align*}
I&\le
\frac{C C_\sharp}{\rho}\left[t^{\rho}-(t-s)^{\rho} + s^{\rho} -2((t-s)^\rho-s^\rho)\right]\le \frac{C C_\sharp}{\rho}(t-s)^\rho.
\end{align*}
This proves \eqref{E:G-t1}.
As for \eqref{E:G-t2}, by a similar reasoning, we have
\begin{align*}
\int_s^t\ud r\int_{\R^d}\ud z \;
Y^2\left(t-r,x-z\right)
&\le
C C_\sharp \int_s^t\ud r\: (t-r)^{2(\beta+\nu-1)-d\beta/\alpha}=\frac{C C_\sharp}{\rho} (t-s)^{\rho}\;.
\end{align*}
This completes the proof of Proposition \ref{P:G-SD}.
\end{proof}

The following lemma is a slight extension of \cite[Lemma 1]{MijenaNane14ST} from the case where $\gamma=1-\beta$ to a general $\gamma$.
\begin{lemma}\label{L:Y^2}
Assume that $d<2\alpha$, $\beta\in (0,2)$, and $\gamma\ge 0$.
Then
\[
\int_{\R^d} Y_{\alpha,\beta,\gamma,d}^2(t,x) \ud x = C_{\sharp} \: t^{2(\beta+\gamma-1)-d\beta/\alpha},
\]
for all $t>0$, where
\begin{align}\label{E:Csharp}
C_{\sharp} :=
\frac{2}{\Gamma(d/2)(2\pi \nu)^{d/2}}\int_0^\infty u^{d-1}E_{\beta,\beta+\gamma}^2(-u^\alpha)\ud u.
\end{align}
\end{lemma}
\begin{proof}
By Plancherel's theorem and \eqref{E:FY},
\begin{align*}
\int_{\R^d} Y(t,x)^2 \ud x
&
=\frac{t^{2(\beta+\gamma-1)}}{(2\pi)^d}
\int_{\R^d} E_{\beta,\beta+\gamma}^2(-2^{-1}\nu t^{\beta}|\xi|^\alpha) \ud \xi\\
&
=
\frac{2\pi^{d/2}}{\Gamma(d/2)}\frac{t^{2(\beta+\gamma-1)}}{(2\pi)^d}
\int_{0}^\infty r^{d-1} E_{\beta,\beta+\gamma}^2(-2^{-1}\nu t^{\beta}r^\alpha) \ud r\\
&=
\frac{2\pi^{d/2}}{\Gamma(d/2)}\frac{t^{2(\beta+\gamma-1)-d\beta/\alpha}}{(2\pi)^d}
(2/\nu)^{d/2}\int_0^\infty u^{d-1}E_{\beta,\beta+\gamma}^2(-u^\alpha)\ud u.
\end{align*}
Note that by the asymptotic property of the Mittag-Leffler function (\cite[Theorem 1.7]{Podlubny99FDE}),
the last integral is finite if $d<2\alpha$.
\end{proof}

The corresponding results to the next Proposition for the SHE, the SFHE, and the SWE can be found in
\cite[Proposition 5.3]{ChenDalang13Heat}, \cite[Proposition 4.7]{ChenDalang14FracHeat}, and \cite[Lemma 3.2]{ChenDalang14Wave}, respectively.
We need some notation: for $\tau>0$, $\alpha>0$ and $(t,x)\in\R_+^*\times\R^d$,
denote
\begin{align*}
B_{t,x,\tau,\alpha}:=\left\{(t',x')\in\R_+^*\times\R^d:\: 0 \le t'\le
t+\tau,\: |x-x'|\le \alpha\right\}.
\end{align*}

\begin{proposition}\label{P:G-Margin}
Suppose that $\beta\in(0,2)$ and $\gamma\in [0,\Ceil{\beta}-\beta]$.
Then for all $(t,x)\in\R_+^*\times\R^d$, there exists a constant $A>0$ such that for all
$(t',x')\in B_{t,x,1/2,1}$ and all $s \in [0,t')$ and $y\in \R^d$ with
$|y|\ge A$, we have that $Y\left(t'-s,x'-y\right) \le Y\left(t+1-s,x-y\right)$.
\end{proposition}

\begin{proof}
Without loss of generality, assume that $\nu=2$.

\vspace{0.5em}\noindent{\bf Case I.~~}
We first prove the case where $\alpha=2$.
The proof here simplifies the arguments of \cite[Proposition 6.1]{Chen14Time}.
Fix $(t,x)\in\R_+^*\times\R^d$.
By the scaling property and the asymptotic property of $Y$, we have that
\[
\frac{Y(t+1-s,x-y)}{Y(t'-s,x'-y)}
\approx \left(\frac{t'-s}{t+1-s}\right)^{f(\beta)}
\frac{|x-y|^a}{|x'-y|^a}\exp\left(
\frac{b|x'-y|^c}{(t'-s)^{\beta c/2}}
-\frac{b|x-y|^c}{(t+1-s)^{\beta c/2}}
\right),
\]
as $|y|\rightarrow\infty$, where  the constants $a$, $b$ and $c$ are defined in \eqref{E:abc},
and
\[
f(\beta)=1+\frac{d\beta}{2}-\beta-\nu + \frac{a\beta}{2}.
\]
Notice that
\begin{align}\label{E_:t-s}
 \frac{t+1-s}{t'-s} = 1+\frac{t+1-t'}{t'-s} \ge
1+ \frac{t+1-t'}{t'} \ge
\frac{t+1}{t+1/2} = 1+ \frac{1}{2t +1} >1\;.
\end{align}
If $f(\beta)\le 0$, then
\[
\left(\frac{t'-s}{t+1-s}\right)^{f(\beta)} = \left(\frac{t+1-s}{t'-s}\right)^{|f(\beta)|}\ge 1.
\]
If $f(\beta)> 0$, then
\[
\left(\frac{t'-s}{t+1-s}\right)^{f\left(\beta\right)}\ge
\left(\frac{t'-s}{t+1}\right)^{\left|f\left(\beta\right)\right|}
= (t+1)^{-\left|f\left(\beta\right)\right|}
\exp\left(\left|f\left(\beta\right)\right|\;\log(t'-s)\right).
\]
The rest arguments are the same as the proof of \cite[Proposition 6.1]{Chen14Time}. We will not repeat here.

\vspace{0.5em}\noindent{\bf Case II.~~} Now we consider the case when $\alpha\in(0,2)$.
By the scaling property and the asymptotic property of $Y$, we have that
\[
\frac{Y(t+1-s,x-y)}{Y(t'-s,x'-y)}
\approx \left(\frac{t'-s}{t+1-s}\right)^{1-2\beta-\nu}
\left(\frac{|x'-y|}{|x-y|}\right)^{d+\alpha},
\]
as $|y|\rightarrow\infty$.
Because $\beta>1/2$ and $\gamma\ge 0$, we see that $1-2\beta-\gamma<0$. Hence, by \eqref{E_:t-s},
\[
\left(\frac{t'-s}{t+1-s}\right)^{1-2\beta-\nu}
=\left(\frac{t+1-s}{t'-s}\right)^{2\beta+\nu-1}
\ge \left(1+\frac{1}{2t+1}\right)^{2\beta+\nu-1}>1.
\]
On the other hand,
\[
\frac{|x'-y|}{|x-y|}\ge
\frac{|y|-|x'|}{|x|+|y|}\ge \frac{|y|-(|x'-x|+|x|)}{|x|+|y|}
\ge
\frac{|y|-1+|x|}{|x|+|y|} \rightarrow 1,
\]
as $|y|\rightarrow \infty$.
Therefore, we can choose a large constant $A$, such that for all $|y|\ge A$ and  all $(t',x')\in B_{t,x,1/2,1}$
and $s\in [0,t']$,
\[
\frac{Y(t+1-s,x-y)}{Y(t'-s,x'-y)}>1.
\]
This completes the proof of Proposition \ref{P:G-Margin}.
\end{proof}

\begin{proposition}\label{P:G-FD}
For all $(t,x)\in\R_+\times\R^d$, $1<\beta<2$ and $\gamma\in [0,2-\beta]$,
we have
\[
\lim_{(t',x')\rightarrow (t,x)} \iint_{\R_+\times\R^d}\ud s\ud y\: \left(
Y\left(t'-s,x'-y\right)
-Y\left(t-s,x-y\right)
\right)^2=0.
\]
\end{proposition}
\begin{proof}
This proposition is a consequence of Proposition \ref{P:G-Margin}.
The proof follows the same arguments as the proof of \cite[Proposition 6.4]{Chen14Time}.
\end{proof}

\subsection{Estimations of the kernel function $\calK$}
Let $G: \R_+\times\R^d\mapsto \R$ with $d\in\bbN$, $d \ge 1$ be a Borel measurable function.

\begin{assumption}\label{A:Upper}
The function $G: \R_+\times\R^d\mapsto \R$ has the following properties:
\begin{enumerate}[(1)]
 \item There is a nonnegative function $\calG(t,x)$, called {\it reference kernel function}, and constants $C_0>0$,
$\sigma<1$ such that
\begin{align}\label{E:BdScSG}
G(t,x)^2\le \frac{C_0}{t^\sigma}\; \calG(t,x)\;,\quad\text{for all $(t,x)\in\R_+\times\R^d$.}
\end{align}
\item The reference kernel function $\calG(t,x)$ satisfies the following {\it sub-semigroup property}: for some constant $C_1>0$,
\begin{align}\label{E:SG}
\int_{\R^d} \ud y\: \calG\left(t,x-y\right)\calG\left(s,y\right)\le
C_1\; \calG\left(t+s,x\right)\;,\quad\text{for all $t, s>0$ and
$x\in\R^d$.}
\end{align}
\end{enumerate}
\end{assumption}

\begin{assumption}\label{A:Low}
The same as Assumption \ref{A:Upper} except that the two ``$\le$'' in \eqref{E:BdScSG} and \eqref{E:SG}
are replaced by ``$\ge$''.
We call the property \eqref{E:SG}  with ``$\le$''  replaced by ``$\ge$'' the {\it super-semigroup property}.
\end{assumption}

\begin{proposition}\label{P:ST-Con}
Under conditions \eqref{E:Dalang} and \eqref{E:CaseA},
the function $Y(t,x)$ satisfies Assumption \ref{A:Upper} with the reference kernel
$\calG_{\alpha,\beta}(t,x)$ defined in \eqref{E:calG}, two nonnegative constants
$C_0$ and $C_1$, depending on $(\alpha,\beta,\gamma,\nu,d)$, and $\sigma$ defined in \eqref{E:sigma}.
\end{proposition}
\begin{proof}
The proof is similar to that of \cite[Proposition 5.8]{Chen14Time}.
We first note that $\sigma<1$ is implied by Dalang's condition \eqref{E:Dalang}; see \eqref{E:sigma<1}.

{\bigskip\bf\noindent Case I.~~} We first consider the case where $\alpha=2$. In this case,
\begin{align*}
\calG_{\alpha,\beta}(t,x)=
\left(4\nu\pi t^\beta\right)^{-d/2}
\exp\left(-\frac{1}{4\nu}\left(t^{-\beta/2}|x|\right)^{\Floor{\beta}+1}\right).
\end{align*}
Notice that
\[
\frac{2}{2-\beta}>\Floor{\beta}+1, \quad\text{for $\beta\in (0,1)\cup (1,2)$},
\]
and when $\beta=1$, the constant $b$ defined in \eqref{E:abc} reduces to $1/(2\nu)$, which is bigger than $1/(4\nu)$. Hence,
by \eqref{E:YAsy} and \eqref{E:YZero}, we see that
\[
\sup_{(t,x)\in\R_+\times\R^d}\frac{Y(t,x)^2}{t^{-\sigma}\calG_{2,\beta}(t,x)}
= \sup_{y\in\R^d}\frac{Y(1,y)^2}{\calG_{2,\beta}(1,y)} =: C_0<\infty.
\]
Note that in the application of \eqref{E:YZero} in the above equations
we have used the fact that Dalang's condition \eqref{E:Dalang} implies $d<2\alpha$.
When $\beta=1$, we see that
\[
\calG_{2,1}(t,x) =
\left(4\nu\pi t\right)^{-d/2}
\exp\left(-\frac{|x|^2}{4\nu t}\right),
\]
and hence, $C_1=1$ and \eqref{E:SG} becomes equality in this case. When $\beta\in (0,1)$,
\[
\calG_{2,\beta}(t,x) =
\left(4\nu\pi t^\beta\right)^{-d/2}
\exp\left(-\frac{|x|}{4\nu t^{\beta/2}}\right)
\le
\prod_{i=1}^d \left(4\nu\pi t^\beta\right)^{-1/2}
\exp\left(-\frac{|x_i|}{4\nu t^{\beta/2}}\right).
\]
Then by Lemma 5.10 of \cite{Chen14Time}, for some nonnegative constant $C_1$,
\[
\int_{\R^d}\calG_{2,\beta}(t-s,x-y)\calG_{2,\beta}(s,y)\ud y\le C_1\: \calG_{2,\beta}(t,x).
\]
When $\beta\in (1,2)$,
\[
\calG_{2,\beta}(t,x) =
\left(4\nu\pi t^\beta\right)^{-d/2}
\exp\left(-\frac{|x|^2}{4\nu t^{\beta}}\right)
=\calG_{2,1}(t^\beta,x).
\]
Hence, by the semigroup property for the heat kernel,
\[
\int_{\R^d}\calG_{2,\beta}(t-s,x-y)\calG_{2,\beta}(s,y)\ud y= \calG_{2,1}((t-s)^\beta+s^{\beta},x)
\le 2^{d(1-\beta)}\calG_{2,\beta}(t,x),
\]
where in the last step we have applied the inequalities:
\[
2^{1-\beta}t^\beta\le (t-s)^\beta+s^\beta\le t^\beta,\qquad\text{for $\beta\in(1,2)$.}
\]
Hence, in this case, $C_1=2^{(1-\beta)d}$.
Therefore, Assumption \ref{A:Upper} is satisfied.

{\bigskip\bf\noindent Case II.~~} We now consider the case where $\alpha\ne 2$. In this case,
\begin{align*}
\calG_{\alpha,\beta}(t,x) =
\frac{c_n \:t^{\beta/\alpha}}{\left(t^{2\beta/\alpha}+|x|^2\right)^{(d+1)/2}}= G_p(t^{\beta/\alpha},x),
\end{align*}
where $G_p(t,x)$ is the Poisson kernel (see \cite[Theorem 1.14]{SteinWeiss71}).
By the scaling property and the asymptotic property of $Y(1,x)$ at $0$ and $\infty$ shown in \eqref{E:YZero}
and \eqref{E:YAsy}, for some nonnegative constant $C$,
\[
Y(t,x)\le \frac{C\: t^{\beta+\gamma-1-d\beta/\alpha}}{\left(1+t^{-2\beta/\alpha}|x|^2\right)^{(d+\alpha)/2}}
\le
\frac{C\: t^{\beta+\gamma-1-d\beta/\alpha}}{\left(1+t^{-2\beta/\alpha}|x|^2\right)^{(d+1)/4}},
\]
where the second inequality is due to $(d+\alpha)/2\ge (d+1)/4$, which is equivalent to $d\ge 1-2\alpha$.
Hence,
\[
\sup_{(t,x)\in\R_+\times\R^d}\frac{Y(t,x)^2}{t^{-\sigma}\calG_{\alpha,\beta}(t,x)}
= \sup_{y\in\R^d}\frac{Y(1,y)^2}{\calG_{\alpha,\beta}(1,y)} =: C_0<\infty.
\]
 Then, it is ready to see that
\[
\calG_{\alpha,\beta} (t,x)= G_p(t^{\beta/\alpha},x)
\]
By the semigroup property of the Poisson kernel, we have that
\begin{align}\label{E_:Poisson}
 \int_{\R^d}\calG_{\alpha,\beta}(t-s,x-y)\calG_{\alpha,\beta}(s,y)\ud y
=G_p\left(s^{\beta/\alpha}+(t-s)^{\beta/\alpha},x\right).
\end{align}
Then use the inequalities
\begin{equation}
\label{E_:IneqPoisson}
\begin{aligned}
  t^{\beta/\alpha}\le & s^{\beta/\alpha}+(t-s)^{\beta/\alpha}\le 2^{1-\beta/\alpha} t^{\beta/\alpha}
  & \text{if $\beta/\alpha\le 1$,} \\
  2^{1-\beta/\alpha} t^{\beta/\alpha} \le & s^{\beta/\alpha}+(t-s)^{\beta/\alpha}\le t^{\beta/\alpha}
  & \text{if $\beta/\alpha> 1$,}
\end{aligned}
\end{equation}
to conclude that for some constant $C_1>0$,
\[
\int_{\R^d}\calG_{\alpha,\beta}(t-s,x-y)\calG_{\alpha,\beta}(s,y)\ud y
\le C_1 G_p\left(t^{\beta/\alpha},x\right) = C_1 \calG_{\alpha,\beta}(t,x).
\]
This completes the proof of Proposition \ref{P:ST-Con}.
\end{proof}

\begin{proposition}\label{P:ST-Con2}
Under \eqref{E:Dalang} and the first two cases of \eqref{E:4cases},
the function $Y(t,x)$ satisfies Assumption \ref{A:Low} with the reference kernel
$\ubar{\calG}_{\alpha,\beta}(t,x)$ defined in \eqref{E:calG-L}, two nonnegative constants
$C_0$ and $C_1$, depending on $(\alpha,\beta,\gamma,\nu,d)$, and
$\sigma$ defined in \eqref{E:sigma}.
\end{proposition}
\begin{proof}
The proof is similar to the proof of the previous proposition.
We have several cases.

{\bigskip\bf\noindent Case I:~}
When $\alpha=2$ and $\beta\in(0,1)$, by the scaling property \eqref{E:YScale}, and
the asymptotics at zero and infinity in \eqref{E:YZero} and \eqref{E:YAsy}, we see that
\begin{align}\label{E_:lowCaseI}
\sup_{(t,x)\in\R_+\times\R^d}\frac{t^{-\sigma} \ubar{\calG}_{2,\beta}(t,x)}{Y(t,x)^2}
= \sup_{y\in\R^d}\frac{\ubar{\calG}_{2,\beta}(1,y)}{Y(1,y)^2} =: \frac{1}{C_0}<+\infty.
\end{align}
By the semigroup property of the heat kernel,
\[
\int_{\R^d}\ubar{\calG}_{\alpha,\beta}(t-s,x-y)\ubar{\calG}_{\alpha,\beta}(s,y)\ud y=
\ubar{\calG}_{\alpha,\beta}\left((s^\beta+(t-s)^{\beta})^{1/\beta},x\right)
\ge 2^{(\beta-1)d/2}\ubar{\calG}_{\alpha,\beta}(t,x),
\]
where the last inequality is due to
\[
t^\beta\le (t-s)^\beta+s^\beta\le 2^{1-\beta} t^\beta,\qquad\text{for $\beta\in(0,1)$.}
\]

{\bigskip\bf\noindent Case II:~}
When $\alpha<2$ and $\beta\in(0,1\vee \alpha)$,
by \eqref{E:YScale} and \eqref{E:YAsy},
\[
\sup_{(t,x)\in\R_+\times\R^d}\frac{t^{-\sigma} \ubar{\calG}_{\alpha,\beta}(t,x)}{Y(t,x)^2}
= \sup_{y\in\R^d}\frac{\ubar{\calG}_{\alpha,\beta}(1,y)}{Y(1,y)^2} =: \frac{1}{C_0}<+\infty.
\]
By the super-semigroup property can be proved in the same way from \eqref{E_:Poisson} and \eqref{E_:IneqPoisson}.
\end{proof}

\subsection{A lemma on the initial data}

\begin{lemma}\label{L:InDt}
For all compact sets $K\subseteq \R_+^*\times\R^d$,
\[
\sup_{(t,x)\in K}\left(\left[1+J_0^2\right]\star \calK\right)(t,x)
<\infty,
\]
under the following two cases:
\begin{enumerate}[(1)]
 \item Both \eqref{E:Dalang} and \eqref{E:CaseA} are satisfied
 and the initial data satisfy \eqref{E:BddInit};
 \item Both \eqref{E:Dalang} and \eqref{E:CaseB} are satisfied
 and the initial data belong to $\calM_{\alpha,\beta}\left(\R\right)$.
\end{enumerate}
\end{lemma}
\begin{proof}
In both cases,  the kernel function $\calK(t,x)$ has the following upper bound
\[
\calK(t,x;\lambda)\le C_1 \calG_{\alpha,\beta}(t,x) \left(t^{-\sigma}+e^{C_2 t}\right).
\]
Part (1) is clear because
\[
\left(\left[1+J_0^2\right]\star \calK\right)(t,x)\le (1+\widehat{C}_t) (1\star\calK)(t,x)
=C_1 (1+\widehat{C}_t) \left(\frac{t^{1-\sigma}}{1-\sigma}+\frac{e^{C_2 t}-1}{C_2}\right),
\]
where $\sigma<1$ (see \eqref{E:sigma<1}).
The proof of part (2) requires more work.
The case when $\alpha=2$ is proved in Lemma 6.7 of \cite{Chen14Time}.
The proof for $\alpha\in (0,2)$ is similar to that of Lemma 4.9 in \cite{ChenDalang14FracHeat}.
Let
\[
 G_{\alpha,\beta,d}(t,x)= \pi^{-d/2} t^{\eta-1} |x|^{-d}
 \FoxH{2,1}{2,3}{\frac{ |x|^\alpha}{2^{\alpha-1}\nu t^\beta}}{(1,1),\:(\eta,\beta)}
 {(d/2,\alpha/2),\:(1,1),\:(1,\alpha/2)}
\]
with $\eta=\Ceil{\beta}$ in case of $Z$ and $\eta=1$ in case of $Z^*$.
Hence, we need only consider $J_0(t,x)=(|\mu|*G(t,\cdot))(x)$.
By the asymptotic properties both at infinity and at zero (see Lemma \ref{L:YAsy} and Remark \ref{R:YZero}),
we have that for all $t\in (0,T]$,
\[
G(t,x)= t^{\eta-1-d\beta/\alpha} G(1,t^{-\beta/\alpha}x)
\le \frac{C\: t^{\eta-1-d\beta/\alpha}}{1+|t^{-\beta/\alpha}x|^{d+\alpha}}
\le \frac{C\: t^{\eta-1-d\beta/\alpha} (1\vee T)^{d\beta}}{1+|x|^{d+\alpha}}.
\]
Thus, for $s\in (0,t]$,
\[
J_0(s,y)\le A C s^{\eta-1-d\beta/\alpha} (1\vee t)^{d\beta},
\]
where
\[
A=\sup_{y\in\R}\int_{\R}|\mu|(\ud y)\frac{1}{1+|x-y|^{1+\alpha}}.
\]
The rest of the proof follows line-by-line the proof of part (2) of Lemma 4.9 in \cite{ChenDalang14FracHeat}.
This completes the proof of Lemma \ref{L:InDt}.
\end{proof}

\subsection{Proof of Theorem \ref{T:ExUni}} \label{S:ExUni}

\begin{proof}[Proof of Theorem \ref{T:ExUni}]

The proof is the same as the proof of \cite[Theorem 3.1]{Chen14Time},
which in turn follows the same six steps as those in the proof of
\cite[Theorem 2.4]{ChenDalang13Heat} with some minor changes:

The proof relies on estimates on the kernel function $\calK(t,x)$, which is given by Proposition \ref{P:ST-Con}.

In the Picard iteration scheme, one needs to check the $L^p(\Omega)$-continuity of the stochastic integral.
This will guarantee that the integrand in the next step is again in $\calP_2$, via \cite[Proposition 3.4]{ChenDalang13Heat}.
Here, the statement of \cite[Proposition 3.4]{ChenDalang13Heat} is still true by replacing in its proof
\cite[Proposition 3.5]{ChenDalang13Heat} by either Proposition \ref{P:G-SD} for the slow diffusion equations or
Proposition \ref{P:G-FD} for the fast diffusion equations, and replacing \cite[Proposition 5.3]{ChenDalang13Heat} by Proposition \ref{P:G-Margin}.

In the first step of the Picard iteration scheme, the following property,
which determines the set of the admissible initial data, needs to be verified:
for all compact sets $K\subseteq \R_+\times\R^d$,
\[
\sup_{(t,x)\in K}\left(\left[1+J_0^2\right]\star
\calK \right) (t,x)<+\infty.
\]
For the SHE, this property is proved in \cite[Lemma 3.9]{ChenDalang13Heat}.
Here, Lemma \ref{L:InDt} gives the desired result with minimal requirements on the initial data.
This property, together with the calculation of  the upper bound on $\calK(t,x)$
in Theorem \ref{T:K-bounds}, guarantees that all the $L^p(\Omega)$-moments of $u(t,x)$ are finite.
This property is also used to establish uniform convergence of the Picard iteration scheme, hence $L^p(\Omega)$--continuity of $(t,x)\mapsto I(t,x)$.

The proof of \eqref{E:SecMom-Lower} is identical to that of the corresponding property in \cite[Theorem 2.4]{ChenDalang13Heat}.
This completes the proof of Theorem \ref{T:ExUni}.
\end{proof}
\subsection{Proof of Theorem \ref{T2:ExUni}}
\label{S2:ExUni}

\begin{proof}[Proof of Theorem \ref{T2:ExUni}]
The proof of Theorem \ref{T2:ExUni} is similar to that for Theorem \ref{T:ExUni}.
Because
\[
\widehat{C}_t= \sup_{(s,x)\in[0,t]\times\R^d} |J_0(s,x)|<\infty,\quad\text{for all $t>0$,}
\]
the Picard iterations in the proof of Theorem 2.4 \cite{ChenDalang13Heat} give the following the moment formula
\[
\Norm{u(t,x)}_p^2 \le 2 J(t,x)^2 + \left[\Vip^2+2 \widehat{C}_t^2\right] \left(1\star \widehat{\calK}_p\right)(t,x).
\]
Note that the function $\left(1\star\calK_p\right)(t,x)$ is a function of $t$ only. For convenience, we denote it as
\begin{align}\label{E:H}
 H(t;\lambda):=\int_0^t\ud s\int_{\R^d}\ud y \: \calK(s,y;\lambda).
\end{align}
Therefore, we need only to prove that $H(t;\lambda)$ is finite, which is proved in Lemma \ref{L:Ht} below.
This completes the proof of Theorem \ref{T2:ExUni}.
\end{proof}

\begin{lemma}\label{L:Ht}
For all $\alpha\in (0,2]$, $\beta\in(0,2)$, $\gamma\ge 0$, and $d\in\bbN$, under Dalang's condition \eqref{E:Dalang},
we have that
\[
H(t;\lambda)\le \exp\left(C \lambda^{\frac{2}{1-\sigma}} \: t\right),
\]
for all $t>0$ and $\lambda\in\R$, where $\sigma$ is defined in \eqref{E:sigma}
and $C$ is some constant depending on $\alpha$, $\beta$, $\gamma$ and $d$.
\end{lemma}
\begin{proof}
By Lemma \ref{L:Y^2},
\[
(1\star \calL_0)(t,x) \le  C_{\sharp}\int_0^t\ud s \: s^{2(\beta+\gamma-1)-d\beta/\alpha}= \frac{C_\sharp s^\theta}{\theta},
\]
where $C_\sharp$ is defined in \eqref{E:Csharp} and $\theta=1-\sigma$.
Note that $\theta>0$ is guaranteed by Dalang's condition \eqref{E:Dalang}.
Now we claim that, for $n\ge 0$,
\begin{align}\label{E:Induction}
(1\star \calL_n)(t,x)\le \frac{C_\sharp^{n+1} \Gamma\left(\theta\right)^{n+1} t^{(n+1)\theta}}{\Gamma((n+1)\theta+1)},
\end{align}
of which the case $n=0$ is just proved.
Assume that \eqref{E:Induction} holds for $n$. By the above calculations, we see that
\[
\left(1\star\calL_{n+1}\right)(t,x)\le
\frac{C_\sharp^{n+2} \Gamma\left(\theta\right)^{n+1}}{\Gamma((n+1)\theta+1)}
\int_0^t \ud s \: (t-s)^{(n+1)\theta} s^{\theta}
=\frac{C_\sharp^{n+2} \Gamma\left(\theta\right)^{n+2} t^{(n+2)\theta}}{\Gamma((n+2)\theta+1)}.
\]
Therefore,
\[
H(t;\lambda) =
\sum_{n=0}^\infty
\lambda^{2(n+1)}(1\star \calL_n)(t,x)
\le
E_{\theta,\theta+1}\left(C_\sharp\Gamma(\theta)\lambda^2 t^\theta \right).
\]
Then apply the asymptotic property of the Mittag-Leffler function (see, e.g., \cite[Theorem 1.3]{Podlubny99FDE}).
This completes the proof of Lemma \ref{L:Ht}.
\end{proof}

\appendix
\section{Appendix: Some properties of the Fox H-functions}
In this section, we follow the notation of \cite{KilbasSaigo04H}.
\begin{definition}\label{D:H}
Let  $m, n, p, q$ be integers such that  $0\leq m\leq q,  0\leq n\leq p$.  Let
  $a_i,  b_i\in \mathbb{C}$ be complex numbers and
  let $\alpha_j, \beta_j$ be positive numbers,  $i=1, 2, \cdots ,  p; j=1, 2, \cdots  ,  q$.
  Let the   set of poles of the gamma functions $\Gamma(b_j+\beta_js)$ doesn't intersect with that of the gamma functions $\Gamma(1-a_i-\alpha_is)$,
namely,
\[
\bigg\{b_{jl}=\frac{-b_j-l}{\beta_j},  l =0, 1, \cdots\bigg\}\bigcap \bigg\{a_{ik}=\frac{1-a_i+k}{\alpha_i},  k=0, 1, \cdots\bigg\}=\emptyset\
\]
for all $i=1, 2, \cdots ,  p$ and $ j=1, 2, \cdots ,  q$.
 Denote $$\mathcal{H}^{mn}_{pq}(s):=\frac{\prod_{j=1}^m \Gamma(b_j+\alpha_js)\prod_{i=1}^n\Gamma(1-a_i-\alpha_is)}{\prod_{i=n+1}^p\Gamma(a_j+\alpha_is)\prod_{j=m+1}^q \Gamma(1- b_j-\alpha_js)}, $$
The {\em Fox {\sl H}-function}
\[
H^{m,n}_{p,q}(z)\equiv H^{m,n}_{p,q}\bigg[z \bigg|\begin{array}{ccc}
(a_1, \alpha_1) & \cdots & (a_p, \alpha_p)\\
(b_1, \beta_1) & \cdots & (b_q, \beta_q)
\end{array} \bigg]
\]
is defined by   the following integral
\begin{equation}\label{E:FoxH}
H^{mn}_{pq}(z)=\frac{1}{2\pi i}\int_L \mathcal{H}^{mn}_{pq}(s) z^{-s} ds\,, \ \ z\in \mathbb{C}\,,
\end{equation}
where an empty product in \eqref{E:FoxH}  means  $1$, and
$L$ in \eqref{E:FoxH} is the infinite contour which separates all the points  $b_{jl}$ to the left and all the points
 $a_{ik}$ to the right of $L$.  Moreover, $L$  has one of the following forms:
\begin{enumerate}[(1)]
\item $L=L_{-\infty}$ is a left loop situated in a horizontal strip starting at point $-\infty+i\phi_1$ and terminating at point $-\infty+i\phi_2$ for some   $-\infty<\phi_1< \phi_2<\infty$
\item $L=L_{+\infty}$ is a right loop situated in a horizontal strip starting at point $+\infty+i\phi_1$ and terminating at point $\infty+i\phi_2$ for some  $-\infty<\phi_1< \phi_2<\infty$
\item $L=L_{i\gamma\infty}$ is a contour starting at point $\gamma-i\infty$ and terminating at point $\gamma+i\infty$ for some  $\gamma\in(-\infty,  \infty)$
\end{enumerate}
\end{definition}
According to \cite[Theorem 1.1]{KilbasSaigo04H}, the integral \eqref{E:FoxH} exists, for example,
when
\begin{align}\label{E:Delta}
\Delta:=\sum_{j=1}^q\beta_j-\sum_{i=1}^p\alpha_i\geq0 \quad\text{and}\quad L=L_{-\infty},
\end{align}
or when
\begin{align}\label{E:a*}
a^*:=\sum_{i=1}^n \alpha_i -\sum_{i=n+1}^p\alpha_i+\sum_{j=1}^m\beta_j-\sum_{j=m+1}^{q}\beta_j\geq 0
\quad\text{and}\quad L=L_{i\gamma \infty}.
\end{align}
The following two parameters of the Fox H-functions \eqref{E:FoxH} will be used in this paper:
\begin{align}\label{E:mu}
\mu=\sum_{j=1}^q b_j -\sum_{i=1}^p a_i + \frac{p-q}{2},
\end{align}
and
\begin{align}\label{E:delta}
\delta=\prod_{i=1}^p\alpha_i^{-\alpha_i}\prod_{j=1}^q \beta_j^{\beta_j}.
\end{align}

\begin{lemma}\label{L:H1/2}
For $b\in\bbC$ and $\beta>0$, there holds the relation
\[
\FoxH{2+m,n}{p,2+q}{z}{(a_i,\alpha_i)_{1,p}}{(b,\beta),(b+1/2,\beta),(b_j,\beta_j)_{1,q}}
=
2^{1-2b}\sqrt{\pi}\: \FoxH{1+m,n}{p,1+q}{4^\beta z}{(a_i,\alpha_i)_{1,p}}{(2b,2\beta),(b_j,\beta_j)_{1,q}}
\]
and
\[
\FoxH{m,n}{p,2+q}{z}{(a_i,\alpha_i)_{1,p}}{(b_j,\beta_j)_{1,q},\:(b,\beta),(b+1/2,\beta)}
=
4^{-b}\pi^{-1/2}\: \FoxH{1+m,n}{p,1+q}{4^{\beta} z}{(a_i,\alpha_i)_{1,p}}{(b_j,\beta_j)_{1,q},\:(2b,2\beta)}.
\]
\end{lemma}
\begin{proof}
Let
\[
\calH(s)=\frac{\prod_{j=1}^m \Gamma(b_j+\beta_j s) \prod_{i=1}^n\Gamma(1-a_i-\alpha_i s) }{\prod_{i=n+1}^p \Gamma(a_i+\alpha_i s) \prod_{j=m+1}^q\Gamma(1-b_j-\beta_j s) }\: .
\]
By the definition of the Fox H-function,
\[
\FoxH{2+m,n}{p,2+q}{z}{(a_i,\alpha_i)_{1,p}}{(b,\beta),(b+1/2,\beta),(b_j,\beta_j)_{1,q}}
=\frac{1}{2\pi i}\int_{\calL}\calH(s)\Gamma(b+\beta s)\Gamma(b+1/2+\beta s) z^{-s}\ud s.
\]
 By the duplication rule of the Gamma function \cite[5.5.5 on p. 138]{NIST2010}
 \begin{align}\label{E:Pi1/2}
 \Gamma(z)\Gamma(z+1/2) = \sqrt{\pi} 2^{1-2z} \Gamma(2z),\quad 2z\ne 0,-1,-2,\dots,
 \end{align}
 we have that
 \[
\FoxH{2+m,n}{p,2+q}{z}{(a_i,\alpha_i)_{1,p}}{(b,\beta),(b+1/2,\beta),(b_j,\beta_j)_{1,q}}
=2^{1-2b}\sqrt{\pi} \frac{1}{2\pi i}\int_{\calL}\calH(s)\Gamma(2b+2\beta s) (4^\beta z)^{-s}\ud s.
\]
Then apply the definition of the Fox H-function.
The second relation can be proved similarly.
\end{proof}
Here are some direct consequences of this lemma:
\begin{align*}
&\FoxH{1+m,n}{1+p,1+q}{z}{(a_i,\alpha_i)_{1,p},\:(b,\beta)}{ (2b,2\beta),\:(b_j,\beta_j)_{1,q}}
=
2^{2b-1}\pi^{-1/2}\:\FoxH{1+m,n}{p,1+q}{4^\beta z}{(a_i,\alpha_i)_{1,p}}{(1/2+b,\beta),\:(b_j,\beta_j)_{1,q}},\\[0.5em]
&\FoxH{1+m,n}{1+p,1+q}{z}{(a_i,\alpha_i)_{1,p},\:(1/2+b,\beta)}{ (2b,2\beta),\:(b_j,\beta_j)_{1,q}}
=
2^{2b-1}\pi^{-1/2}\:\FoxH{1+m,n}{p,1+q}{4^\beta z}{(a_i,\alpha_i)_{1,p}}{(b,\beta),\:(b_j,\beta_j)_{1,q}},\\[0.5em]
&\FoxH{m,1+n}{1+p,1+q}{z}{(1/2+b,\beta),\:(a_i,\alpha_i)_{1,p}}{(b_j,\beta_j)_{1,q},\: (2b,2\beta)}
=
4^{b}\pi^{1/2}\:\FoxH{m,n}{p,1+q}{4^{-\beta} z}{(a_i,\alpha_i)_{1,p}}{(b_j,\beta_j)_{1,q},\:(b,\beta)},\\[0.5em]
&\FoxH{m,1+n}{1+p,1+q}{z}{(1/2+b,\beta),\:(a_i,\alpha_i)_{1,p}}{(b_j,\beta_j)_{1,q},\: (2b,2\beta)}
=
4^{b}\pi^{1/2}\:\FoxH{m,n}{p,1+q}{4^{-\beta} z}{(a_i,\alpha_i)_{1,p}}{(b_j,\beta_j)_{1,q},\:(b,\beta)}.
\end{align*}

\bigskip

\begin{remark}
In \cite{Chen14Time}, the Green function $G_\beta(t,x)$, which corresponds to $Y_{2,\beta,\Ceil{\beta}-\beta,1}(t,x)$,
is represented using the two-parameter Mainardi function of order
$\lambda\in [0,1)$ (see \eqref{E:2p-Mainardi}).
By the series expansion of the Fox H-function (\cite[Theorem 1.3]{KilbasSaigo04H}, which requires that $\Delta=1-\lambda>0$), one can see that
\begin{align}\label{E:M-H}
M_{\lambda,\mu} (z) =
z^{-1}\FoxH{1,0}{1,1}{z}{(\mu,\lambda)}{(1,1)},\quad \lambda\in [0,1).
\end{align}
By Property 2.4 of \cite{KilbasSaigo04H}, the above relation can also be written as
\begin{align}\label{E:M-H2}
\FoxH{1,0}{1,1}{z^2}{(\mu,\lambda)}{(1,2)}=\frac{z}{2}\: M_{\lambda/2,\mu} (z),\quad \lambda\in [0,1).
\end{align}
\end{remark}


\begin{remark}
Another commonly used special function in this setting, such as in \cite{Pskhu09}, is {\it Wright's function}
\cite{Wright34AsymBessel,Wright33Coeff,Wright40GenBessel} (see also \cite[Appendix F]{Mainardi10Book}):
\begin{align}
 \phi(\lambda,\mu;z):=\sum_{k=0}^\infty \frac{z^k}{k! \Gamma(\mu+\lambda k)},\quad \text{for $\lambda>-1$, $\mu\in\bbC$.}
\end{align}
We adopt the notation $\phi$ that is used by E.~M. Wright in his original papers.
By (2.9.29) and Property 2.5 of \cite{KilbasSaigo04H},
\begin{align}\label{E:W-H}
 \phi(\lambda,\mu;z) =
 \begin{cases}
  z^{-1} \FoxH{1,0}{0,2}{z}{\midrule}{(1,1),\:(1+\lambda-\mu,\lambda)} & \text{if $\lambda> 0$},\\[1em]
  z^{-1} \FoxH{1,0}{1,1}{z}{(\mu-\lambda,-\lambda)}{(1,1)} & \text{if $\lambda\in (-1,0]$}.\\
 \end{cases}
\end{align}
Comparing \eqref{E:M-H} and \eqref{E:W-H}, we see that
\begin{align}\label{E:M-Wright}
M_{\lambda,\mu}(z) = \phi(-\lambda,\mu-\lambda;z),\quad\text{for $\lambda\in (0,1]$.}
\end{align}
\end{remark}

\bigskip
The following theorem is a simplified version of Theorems 2.9 and 2.10 in \cite{KilbasEtc06}, which is
sufficient for our use in the proof of Theorem \ref{T:NonY}.
%

\begin{theorem}\label{T:HConvH}
Let $(a_1^*,\Delta_1,\mu_1)$ and $(a_2^*,\Delta_2,\mu_2)$ be the constants $(a^*,\Delta,\mu)$ defined in
\eqref{E:a*}, \eqref{E:Delta} and \eqref{E:mu} for the following two Fox H-functions:
\[
 \FoxH{m,n}{p,q}{x}{(a_i,\alpha_i)_{1,p}}{(b_j,\beta_j)_{1,q}}
 \quad\text{and}\quad
 \FoxH{M,N}{P,Q}{x}{(d_i,\delta_i)_{1,P}}{(c_j,\gamma_j)_{1,Q}},
\]
respectively. Denote
\begin{align*}
&A_1 = \min_{1\le i\le n} \frac{1-\Re(a_i)}{\alpha_i},\quad
B_1 = \min_{1\le j\le m} \frac{\Re(b_j)}{\beta_j},\\
&A_2 = \min_{1\le j\le M} \frac{\Re(c_j)}{\gamma_j},\quad
B_2 = \min_{1\le i\le N} \frac{1-\Re(d_i)}{\delta_i},
\end{align*}
with the convention that $\min(\phi)=+\infty$.
If either of the following four conditions holds
\begin{enumerate}[(1)]
 \item $a_1^*> 0$ and $ a_2^*> 0$;
 \item $a_1^*=\Delta_1=0$, $\Re(\mu_1)<-1$ and $a_2^*>0$;
 \item $a_2^*=\Delta_2=0$, $\Re(\mu_2)<-1$ and $a_1^*>0$;
 \item $a_1^*=\Delta_1=0$, $\Re(\mu_1)<-1$ and $a_2^*=\Delta_2=0$, $\Re(\mu_2)<-1$,
\end{enumerate}
and if
\begin{align}\label{E:HConvH}
A_1+B_1>0,\quad A_2+B_2>0,\quad A_1+A_2>0,\quad B_1+B_2>0,
\end{align}
then, for all $z>0$, $x\in\R$,
\begin{align*}
&\FoxH{m+M,n+N}{p+P,q+Q}{z x}{(a_i,\alpha_i)_{1,n},\: (d_i,\delta_i)_{1,P},\: (a_i,\alpha_i)_{n+1,p}}
 {(b_j,\beta_j)_{1,m},\: (c_j,\gamma_j)_{1,Q},\: (b_j,\beta_j)_{m+1,q}}
 \\
&\hspace{4em} =\int_0^\infty
 \FoxH{m,n}{p,q}{z t}{(a_i,\alpha_i)_{1,p}}{(b_j,\beta_j)_{1,q}}
 \FoxH{M,N}{P,Q}{\frac{x}{t}}{(d_i,\delta_i)_{1,P}}{(c_j,\gamma_j)_{1,Q}}\frac{\ud t}{t}.
\end{align*}
\end{theorem}
\begin{proof}
By Property 2.3 of \cite{KilbasEtc06},
\begin{align}\label{E_:HH}
 \FoxH{M,N}{P,Q}{\frac{x}{t}}{(d_i,\delta_i)_{1,P}}{(c_j,\gamma_j)_{1,Q}}
 =
 \FoxH{N,M}{Q,P}{\frac{t}{x}}{(1-c_j,\gamma_j)_{1,Q}}{(1-d_i,\delta_i)_{1,P}}.
\end{align}
If condition (1) holds, one can apply Theorem 2.9 of \cite{KilbasEtc06} with $\eta=0$, $\sigma=1$, $w=1/x$, and with the following replacements:
$N\rightarrow M$, $M\rightarrow N$,
$P\rightarrow Q$, $Q\rightarrow P$,
$c_j\rightarrow 1-c_j$, $d_i\rightarrow 1-d_i$.
If either of conditions (2)--(4) holds, we apply Theorem 2.10 of \cite{KilbasEtc06}
in the same way.
Note that the parameters $\mu$ for both Fox H-functions in \eqref{E_:HH} are equal.
\end{proof}

\subsection{Proof of Lemma \ref{L:HAt0}}
\label{SS:HAt0}
\begin{proof}[Proof of Lemma \ref{L:HAt0}]
Let
\[
f(x)=\FoxH{2,1}{2,3}{x}{(1,1),\;(\eta,\beta)}{(d/2,\alpha/2),\:(1,1),\:(1,\alpha/2)}.
\]
Then $g(x)=x^{-d}f(x^\alpha)$.
Let
\[
H_{d,\alpha,\beta,\eta}(s):=\frac{\Gamma(d/2+\alpha s/2)\Gamma(1+s)\Gamma(-s)}{\Gamma(\eta+\beta s)\Gamma(-\alpha s/2)}.
\]
Denote the poles of $\Gamma(1+s)$ and $\Gamma(d/2+\alpha s/2)$ by
\[
A:=\{-(1+k): k=0,1,2,\cdots\}\quad\text{and}\quad
B:=\left\{-\frac{2l+d}{\alpha}: l=0,1,2,\cdots\right\},
\]
respectively.
According to the definition of Fox H-function, to calculate the asymptotic at zero,
we need to calculate the residue of $H_{d,\alpha,\beta,\eta}(s) z^{-s}$ at the rightmost poles in $A\cup B$.
Because $a^*=2-\beta>0$, all the nigh cases are covered by either (1.8.1) or (1.8.2) of \cite{KilbasEtc06}.
The notation $h_{jl}^*$ below follows from (1.3.5) of \cite{KilbasEtc06}.

{\bf\bigskip\noindent Case 1.~~} Assume that $\eta\ne \beta$ and $d/\alpha>1$.
In this case, the rightmost residue in $A\cup B$ is at $s=-1$ and it is a simple pole.
Hence,
\[
h_{20}^*=\frac{\Gamma\left((d-\alpha)/2\right)}{\Gamma(\eta-\beta)\Gamma(\alpha/2)}>0,
\]
and
\[
f(x)=h_{20}^* x +O\left(x^{\min(2,d/\alpha)}\right),\quad\text{as $x\rightarrow 0_+$}.
\]

{\bf\bigskip\noindent Case 2.~~} Assume that $\eta\ne \beta$ and $d/\alpha=1$.
The rightmost residue in $A\cup B$ is at $s=-1$ and it is of order two. Hence,
\begin{align*}
\Res_{s=-1}(H_{d,d,\beta,\eta}(s)x^{-s}) &= \lim_{s\rightarrow-1}\left[(s+1)^2H_{d,d,\beta,\eta}(s) x^{-s}\right]'\\
&=\lim_{s\rightarrow-1}\left(\left[(s+1)^2H_{d,d,\beta,\eta}(s)\right]' - (s+1)^2H_{d,d,\beta,\eta}(s)\log x\right) x^{-s}\\
&= C x -\frac{1}{\Gamma(\eta-\beta)\Gamma(1+d/2)} x \log x,
\end{align*}
where we have used the fact that $\Gamma(x)$ has simple poles at $x=-n$, $n=0,1,\dots$, with residue $\frac{(-1)^n}{n!}$.
Therefore,
\[
f(x) = -\frac{1}{\Gamma(\eta-\beta)\Gamma(1+d/2)} x \log x + O(x),\quad\text{as $x\rightarrow 0_+$}.
\]

{\bf\bigskip\noindent Case 3.~~} Assume that $\eta\ne \beta$ and $d/\alpha<1$.
The rightmost residue in $A\cup B$ is at $s=-d/\alpha$ and it is a simple pole. Hence,
\begin{align}\label{E:h10}
h_{10}^* =\frac{2}{\alpha} \frac{\Gamma(1-d/\alpha)\Gamma(d/\alpha)}{\Gamma(\eta-d\beta/\alpha)\Gamma(d/2)}>0,
\end{align}
where the nonnegativity is due to the fact that $\eta\ge \beta>\beta d/\alpha$. Therefore,
\[
f(x)= h_{10}^* x^{d/\alpha} + O(x^{\min((d+2)/\alpha,1)})
=h_{10}^* x^{d/\alpha} +O(x),\quad\text{as $x\rightarrow 0_+$}.
\]

{\bf\bigskip\noindent Case 4.~~} Assume that $\eta=\beta=1$.
By Property 2.2 of \cite{KilbasEtc06},
\[
f(x)=\FoxH{1,1}{1,2}{x}{(1,1)}{(d/2,\alpha/2),\:(1,\alpha/2)}.
\]
Hence,
\[
h_{10}^* = \frac{2\Gamma(d/\alpha)}{\alpha\Gamma(d/2)}\ne 0,
\]
and
\[
f(x)= h_{10}^* x^{d/\alpha} + O(x^{(d+2)/\alpha}),\quad\text{as $x\rightarrow 0_+$.}
\]
%

{\bf\bigskip\noindent Case 5.~~} Assume that $\eta=\beta\ne 1$ and $d/\alpha>2$.
The rightmost residue in $A\cup B$ is at $s=-1$, but this residue is vanishing because $\lim_{s\rightarrow-1}1/\Gamma(\beta+\beta s)=0$.
The rightmost nonvanishing residue in $A\cup B$ is at $s=-2$ and it is a simple pole. Hence,
\[
h_{21}^*=- \frac{\Gamma((d-2\alpha)/2)}{\Gamma(-\beta)\Gamma(\alpha)}.
\]
Hence,
\[
f(x)= h_{21}^* x^{2} + O(x^{\min(3,d/\alpha)}),\quad\text{as $x\rightarrow 0_+$.}
\]

{\bf\bigskip\noindent Case 6.~~} Assume that $\eta=\beta\ne 1$ and $d/\alpha=2$.
As in Case 6, the rightmost nonvanishing residue in $A\cup B$ is at $s=-2$,
and it is of order two. Then
\begin{align*}
\Res_{s=-2}(H_{d,d/2,\beta,\beta}(s)x^{-s}) &= \lim_{s\rightarrow-2}\left[(s+2)^2H_{d,d/2,\beta,\beta}(s) x^{-s}\right]'\\
&=\lim_{s\rightarrow-2}\left(\left[(s+2)^2H_{d,d/2,\beta,\beta}(s)\right]' - (s+2)^2H_{d,d/2,\beta,\beta}(s)\log x\right) x^{-s}\\
&= C x^2 +\frac{2}{\Gamma(-\beta)\Gamma(1+d/2)} x^2 \log x.
\end{align*}
Therefore,
\[
f(x) = \frac{2}{\Gamma(-\beta)\Gamma(1+d/2)} x^2 \log x + O(x^3),\quad\text{as $x\rightarrow 0_+$}.
\]

{\bf\bigskip\noindent Case 7.~~} Assume that $\eta=\beta\ne 1$ and $d/\alpha\in (1,2)$.
As in Case 6, because $h_{20}^*\equiv 0$, the rightmost nonvanishing residue in $A\cup B$ is at $s=-d/\alpha$,
and it is a simple pole. Hence,
\[
f(x) =h_{10}^* x^{d/\alpha} + O(x^2),\quad\text{as $x\rightarrow 0_+$,}
\]
where $h_{10}^*$ is defined in \eqref{E:h10} with $\eta$ replaced by $\beta$.

{\bf\bigskip\noindent Case 8.~~} Assume that $\eta=\beta\ne 1$ and $d/\alpha=1$.
The rightmost nonvanishing residue in $A\cup B$ is at $s=-1$, and it is of order two.
Hence,
\begin{align*}
\Res_{s=-1}(H_{d,d,\beta,\beta}(s)x^{-s}) &= \lim_{s\rightarrow-1}\left[(s+1)^2H_{d,d,\beta,\beta}(s) x^{-s}\right]'\\
&=\lim_{s\rightarrow-1}\big[\calH_1(s)' \calH_2(s)+\calH_1(s)\calH_2(s)' - \calH_1(s)\calH_2(s)\log x\big] x^{-s},
\end{align*}
where
\[
\calH_1(s)=(s+1)^2\Gamma((1+s)d/2)\Gamma(1+s)\quad\text{and}
\quad
\calH_2(s)= \frac{\Gamma(-s)}{\Gamma(\beta+\beta s)\Gamma(-d s /2)}.
\]
As calculated in the proof of Lemma 7.1 of \cite{CHHH15Time}, we have that
\begin{align*}
\calH_1(-1) &= \lim_{s\rightarrow -1} \calH_1^*(s)=\frac{2}{d}= \lim_{s\rightarrow-1}\frac{(1+s)^2}{((1+s)d/2)(1+s)}=\frac{2}{d},\\
\calH_2(-1) &=\lim_{s\rightarrow -1} \calH_2^*(s) = 0,\\
\left.\frac{\ud }{\ud s}\calH_2(s)\right|_{s=-1} &= \lim_{s\rightarrow -1} \frac{\Gamma(-s)}{\Gamma(-ds/2)} \left(\frac{1}{\Gamma(\beta(1+ s))}\right)'\\
&=\frac{\Gamma(1)}{\Gamma(d/2)}\lim_{s\rightarrow -1}- \frac{\psi(\beta(1+ s))}{\Gamma(\beta(1+ s))}\\
&=\frac{\beta}{\Gamma(d/2)}\,,
\end{align*}
where $\psi(z)$ is the digamma function and the last limit is due to (5.7.6) and (5.7.1) of \cite{NIST2010}.
Therefore,
\begin{align*}
\Res_{s=-1}(H_{d,d,\beta,\beta}(s)x^{-s}) = \frac{\beta}{\Gamma(1+d/2)}\: x,
\end{align*}
and
\[
f(x)=  \frac{\beta}{\Gamma(1+d/2)}\: x + O(x^2),\quad\text{as $x\rightarrow 0_+$.}
\]

{\bf\bigskip\noindent Case 9.~~} Assume that $\eta=\beta\ne 1$ and $d/\alpha<1$.
The first nonvanishing residue in $A\cap B$ is at $s=-d/\alpha$ and it is a simple pole. Hence,
\[
f(x)=h_{10}^* x^{d/\alpha} +O(x),\quad\text{as $x\rightarrow 0_+$,}
\]
where $h_{10}^*$ is defined in \eqref{E:h10} with $\eta$ replaced by $\beta$.
This completes the whole proof of Lemma \ref{L:HAt0}.
\end{proof}
\addcontentsline{toc}{section}{Bibliography}

\begin{thebibliography}{10}


\bibitem{BertiniCancrini94Intermittence}
L.~Bertini and N.~Cancrini.
\newblock The stochastic heat equation: {F}eynman-{K}ac formula and intermittence.
\newblock {\em J. Statist. Phys.},  78 (1995), no. 5-6, 1377--1401.

\bibitem{CarmonaMolchanov94}
R.~A. Carmona and S.~A. Molchanov.
\newblock Parabolic {A}nderson problem and intermittency.
\newblock {\em Mem. Amer. Math. Soc.}, 108(518): viii+125, 1994.

\bibitem{LeChen13Thesis}
L.~Chen.
\newblock {\em Moments, intermittency, and growth indices for nonlinear stochastic PDE's with rough initial conditions}.
\newblock PhD thesis, \'Ecole Polytechnique F\'ed\'erale de Lausanne, 2013.

\bibitem{Chen14Time}
L.~Chen.
\newblock Nonlinear stochastic time-fractional diffusion equations on R: moments, H\"older regularity and intermittency.
\newblock {\em Submitted}, arXiv:1410.1911, 2014.

\bibitem{ChenDalang13Holder}
L.~Chen and R.~C. Dalang.
\newblock H\"older-continuity for the nonlinear stochastic heat equation with
  rough initial conditions.
\newblock {\em Stoch. Partial Differ. Equ. Anal. Comput.} 2(2014), no. 3, 316--352.

\bibitem{ChenDalang14Wave}
L.~Chen and R.~C. Dalang.
\newblock Moment bounds and asymptotics for the stochastic wave equation.
\newblock {\em Stochastic Process. Appl.} 125 (2015), no. 4, 1605--1628.

\bibitem{ChenDalang13Heat}
L.~Chen and R.~C. Dalang.
\newblock Moments and growth indices for nonlinear stochastic heat equation with rough initial conditions.
\newblock {\em Ann. Probab.,}  (to appear), 2015.

\bibitem{ChenDalang14FracHeat}
L.~Chen and R.~C. Dalang.
\newblock Moments, intermittency, and growth indices for the nonlinear fractional stochastic heat equation.
\newblock {\em Stoch. Partial Differ. Equ. Anal. Comput.}, to appear, arXiv:1409.4305, 2014.

\bibitem{CHHH15Time}
L.~Chen, Y.~Hu, G.~Hu, and J.~Huang.
\newblock Stochastic time-fractional diffusion equations on $\R^d$.
\newblock {\em Submitted}, arXiv:1508.00252, 2015.


\bibitem{ChenKim14Comparison}
L.~Chen and K.~Kim.
\newblock On comparison principle and strict positivity of solutions to the nonlinear stochastic fractional heat equations.
\newblock {\em Ann. Inst. Henri Poincaré Probab. Stat.}, accepted pending revision, arXiv:1410.0604, 2015.


\bibitem{ChenKimKim15}
Z.~Q.~Chen, K.~H.~Kim, P.~Kim.
\newblock Fractional time stochastic partial differential equations.
\newblock {\it Stochastic Process. Appl.} 125 (2015), no. 4, 1470--1499.

\bibitem{ConusEct12Initial}
D.~Conus, M.~Joseph, D.~Khoshnevisan, and S.-Y. Shiu.
\newblock Initial measures for the stochastic heat equation.
\newblock {\em Ann. Inst. Henri Poincaré Probab. Stat.}, 50 (2014), no. 1, 136--153.

\bibitem{ConusEtal13Wave}
D.~Conus, M.~Joseph, D.~Khoshnevisan, and S.-Y. Shiu.
\newblock Intermittency and chaos for a nonlinear stochastic wave equation in dimension 1.
\newblock In {\em Malliavin Calculus and Stochastic Analysis}, pages 251--279.  Springer, 2013.


\bibitem{Dalang99Extending}
R.~C. Dalang.
\newblock Extending the martingale measure stochastic integral with
  applications to spatially homogeneous s.p.d.e.'s.
\newblock {\em Electron. J. Probab.}, 4 (1999), no. 6, 29 pp.

\bibitem{DebbiDozzi05On}
L.~Debbi and M.~Dozzi.
\newblock On the solutions of nonlinear stochastic fractional partial
  differential equations in one spatial dimension.
\newblock {\em Stochastic Process. Appl.}, 115(11):1764--1781, 2005.


\bibitem{Die04}
K.~Diethelm.
\newblock {\em The analysis of fractional differential equations}, volume 2004 of {\em Lecture Notes in Mathematics}.
\newblock Springer-Verlag, Berlin, 2010.


%

\bibitem{EK04}
S.~Eidelman and A.~Kochubei.
\newblock Cauchy problem for fractional diffusion equations.
\newblock {\em J. Differ. Equations} 199(2), 211-255, 2004.


\bibitem{HuHu15}
G.~Hu and Y.~Hu.
\newblock Fractional diffusion in Gaussian noisy environment.
\newblock {\em Mathematics}, to appear, 2015.



\bibitem{FoondunKhoshnevisan08Intermittence}
M.~Foondun and D.~Khoshnevisan.
\newblock Intermittence and nonlinear parabolic stochastic partial differential
  equations.
\newblock {\em Electron. J. Probab.}, 14 (2009), no. 21, 548--568.




\bibitem{KilbasSaigo04H}
A.~A.~Kilbas. and M.~Saigo.
\newblock {\em H-transforms: theory and applications.}
\newblock Analytical Methods and Special Functions, 9. Chapman \& Hall/CRC, Boca Raton, FL, 2004.

\bibitem{KilbasEtc06}
A.~A.~Kilbas, H.~M.~Srivastava, J.~J.~Trujillo.
\newblock {\em Theory and applications of fractional differential equations.}
\newblock North-Holland Mathematics Studies, 204. Elsevier Science B.V., Amsterdam, 2006.



\bibitem{Koc90}
A. Kochube\u {\i}.
\newblock Diffusion of fractional order. (Russian)
\newblock {\em Differentsial'nye Uravneniya} 26(4), 660--670, 733--734, 1990;
\newblock translation in {\em Differential Equations } 26 (1990), no. 4, 485--492.

%

\bibitem{Mainardi10Book}
F.~Mainardi.
\newblock {\em Fractional calculus and waves in linear viscoelasticity}.
\newblock Imperial College Press, London, 2010.


\bibitem{MainardiEtc01Fundamental}
F.~Mainardi, Y.~Luchko, and G.~Pagnini.
\newblock The fundamental solution of the space-time fractional diffusion equation.
\newblock {\em Fract. Calc. Appl. Anal.}, 4(2): 153--192, 2001.



\bibitem{MijenaNane14Int}
J.~B.~Mijena and E.~Nane.
\newblock Intermittence and time fractional stochastic partial differential equations.
\newblock {\em preprint}, arXiv:1409.7468, 2014.

\bibitem{MijenaNane14ST}
J.~B.~Mijena and E.~Nane.
\newblock Space-time fractional stochastic partial differential equations.
\newblock {\em Stochastic Process. Appl.} 125 (2015), no. 9, 3301--3326.



\bibitem{NIST2010}
F.~W.~J. Olver, D.~W. Lozier, R.~F. Boisvert, and C.~W. Clark, editors.
\newblock {\em N{IST} handbook of mathematical functions}.
\newblock U.S. Department of Commerce National Institute of Standards and Technology, Washington, DC, 2010.

\bibitem{Pskhu09}
A.~V.~Pskhu.
\newblock The fundamental solution of a diffusion-wave equation of fractional order. (Russian),
\newblock {\it Izv. Ross. Akad. Nauk Ser. Mat.} 73 (2009), no. 2, 141--182; translation in
{\it Izv. Math.} 73 (2009), no. 2, 351--392



\bibitem{Podlubny99FDE}
I.~Podlubny.
\newblock {\em Fractional differential equations},
 volume 198 of {\em  Mathematics in Science and Engineering}.
\newblock Academic Press Inc., San Diego, CA, 1999.


\bibitem{SamkoKilbasMarichev93}
S.~G.~Samko, A.~A.~Kilbas, O.~I.~Marichev.
\newblock {\em Fractional integrals and derivatives. Theory and applications.}
\newblock Edited and with a foreword by S. M. Nikol{$'$}ski{\u\i},
Translated from the 1987 Russian original,
Revised by the authors.
\newblock Gordon and Breach Science Publishers, Yverdon, 1993


\bibitem{Schneider96CM}
W.~R. Schneider.
\newblock Completely monotone generalized {M}ittag-{L}effler functions.
\newblock {\em Exposition. Math.}, 14(1):3--16, 1996.

\bibitem{SteinWeiss71}
E.~Stein, G.~Weiss.
\newblock {\it Introduction to Fourier analysis on Euclidean spaces.}
\newblock Princeton Mathematical Series, No. 32. Princeton University Press, Princeton, N. J., 1971.



\bibitem{Walsh86}
J.~B. Walsh.
\newblock An introduction to stochastic partial differential equations.
\newblock In {\em \'{E}cole d'\'et\'e de probabilit\'es de {S}aint-{F}lour, {XIV}---1984}, volume 1180 of {\em Lecture Notes in Math.}, pages 265--439. Springer, Berlin, 1986.

\bibitem{Widder41LaplaceTr}
D.~V. Widder.
\newblock {\em The {L}aplace transform}.
\newblock Princeton Mathematical Series, v. 6. Princeton University Press, Princeton, N. J., 1941.



\bibitem{Wright33Coeff}
E.~M.~Wright.
\newblock On the Coefficients of Power Series Having Exponential Singularities.
\newblock {\em J. London Math. Soc.} 8:1, (1933), 71--79.

\bibitem{Wright34AsymBessel}
E.~M.~Wright.
\newblock The asymptotic expansion of the generalized {Bessel} function.
\newblock {\em Proc. London Math. Soc.}, 38, (1934), 257--270.

\bibitem{Wright40GenBessel}
E.~M.~Wright.
\newblock The generalized Bessel function of order greater than one.
\newblock {\em Quart. J. Math., Oxford Ser.} 11, (1940), 36--48.

\end{thebibliography}

\def\polhk#1{\setbox0=\hbox{#1}{\ooalign{\hidewidth
  \lower1.5ex\hbox{`}\hidewidth\crcr\unhbox0}}} \def\cprime{$'$}
  \def\cprime{$'$}

\vspace{3em}
\hfill\begin{minipage}{0.55\textwidth}
{\bf Le CHEN}, {\bf Yaozhong HU}, {\bf David NUALART}\\[0.2em]
Department of Mathematics\\
University of Kansas\\
405 Snow Hall, 1460 Jayhawk Blvd,\\
Lawrence, Kansas, 66045-7594, USA.\\
E-mails: \: \url{chenle, yhu, nualart@ku.edu}
\end{minipage}

\end{document}